\documentclass[letterpaper]{amsart}
\usepackage{amssymb}
\usepackage{amsmath,amscd}
\usepackage{url}

\newtheorem{theorem}{Theorem}[section]
\newtheorem{lemma}[theorem]{Lemma}
\newtheorem{proposition}[theorem]{Proposition}
\newtheorem{corollary}[theorem]{Corollary}
\newtheorem*{theorem*}{Theorem}

\theoremstyle{definition}
\newtheorem{definition}[theorem]{Definition}
\newtheorem{remark}[theorem]{Remark}
\newtheorem{example}[theorem]{Example}

\newtheorem*{remark*}{Remark}

\newcommand{\tol}{\mathfrak{c}}

\newcommand{\euO}{\mathfrak O}
\newcommand{\euP}{\mathfrak P}
\newcommand{\euD}{\mathfrak D}
\newcommand{\euA}{\mathfrak A}
\newcommand{\euM}{\mathfrak M}
\newcommand{\eub}{\mathfrak b}

\newcommand{\eua}{\mathfrak a}

\newcommand{\A}{\mathcal{A}}

\newcommand{\DD}{\mathcal{D}}
\newcommand{\EE}{\mathcal{E}}

\newcommand{\bZ}{\mathbb Z}

\newcommand{\Sp}{\mathbb{S}_{p}}
\newcommand{\Spn}{\mathbb{S}_{p^n}}

\newcommand{\LRA}{\Leftrightarrow}

\newcommand{\Gal}{\mathrm{Gal}}
\newcommand{\End}{\mathrm{End}}

\newenvironment{pf-IGMS}{{\noindent \em Proof of Theorem
    \ref{IGMS}. }}{\qed \newline} 
\newenvironment{pf-gencount}{{\noindent \em Proof of Theorem
    \ref{gencount}. }}{\qed \newline} 

\begin{document}
\title[Scaffolds and Galois module structure]{Scaffolds and
  generalized integral Galois module structure}

\author{Nigel P.~Byott,  Lindsay N.~Childs and G.~Griffith Elder}

\address{Department of Mathematics, University of Exeter, Exeter 
EX4 4QF U.K.}  \email{N.P.Byott@exeter.ac.uk}

\address{Department of Mathematics and Statistics, University at
  Albany, Albany, NY 12222 U.S.A.}
\email{lchilds@albany.edu}

\address{Department of Mathematics, University of Nebraska at Omaha, Omaha NE 68182-0243 U.S.A.}  
\email{elder@unomaha.edu}

\date{\today}

\bibliographystyle{amsalpha}

\begin{abstract}
Let $L/K$ be a finite, totally ramified $p$-extension of complete
local fields with residue fields of characteristic $p > 0$, and let
$A$ be a $K$-algebra acting on $L$. We define the concept of an
$A$-scaffold on $L$, thereby extending and refining the notion of
a Galois scaffold considered in several previous papers, where $L/K$ was
Galois and $A=K[G]$ for $G=\Gal(L/K)$. When a suitable
$A$-scaffold exists, we show how to answer questions generalizing
those of classical integral Galois module theory.  We give a necessary
and sufficient condition, involving only numerical parameters, for a given
fractional ideal to be free over its associated order in $A$. We also show
how to determine the number of generators required when it is not
free, along with the embedding dimension of the associated order. In
the Galois case, the numerical parameters are the ramification breaks
associated with $L/K$.  We apply these results to biquadratic Galois
extensions in characteristic 2, and to totally and weakly ramified
Galois $p$-extensions in characteristic $p$. We also
apply our results to the non-classical situation where $L/K$ is a
finite primitive purely inseparable extension of arbitrary exponent
that is acted on, via a higher derivation (but in many different
ways), by the divided power $K$-Hopf algebra. 
\end{abstract}

\maketitle

\section{Introduction}\label{intro}

Let $K$ be a local field with residue field of characteristic $p>0$,
and let $L$ be a finite Galois extension of $K$ with Galois group
$G$. We write $\euO_K$, $\euO_L$ for the valuation rings of $K$, $L$,
respectively, and $\euP_K$, $\euP_L$ for their maximal ideals. Then
$\euO_L$ is a module over the integral group ring $\euO_K[G]$.  By
Noether's criterion \cite{noether}, it is a free module if and only if
the extension $L/K$ is at most tamely ramified. In order to study
integral Galois module structure for wildly ramified extensions,
H.W.~Leopoldt \cite{leopoldt} introduced the associated order
  $$ \euA_{L/K}=\{\alpha\in K[G] :\alpha \cdot \euO_L\subseteq \euO_L\} $$ of
$\euO_L$ in the group algebra $K[G]$. Over the
last 50 years, many authors have investigated, in various situations,
when $\euO_L$ is free as a module over $\euA_{L/K}$, or, more generally,
when a fractional ideal $\euP_L^h$ of $\euO_L$ is free as a module
over its associated order in $K[G]$; see for instance \cite{jacobinski, ferton,
  berge:dihedral, martel, taylor:fg, childs-moss, miyata:cyclic2,
  bondarko:ideals, aiba, desmit:2}. For a comprehensive overview of
this area, and a far more extensive bibliography, we refer the reader
to the survey \cite{thomas:survey}.

Our goal here is to give a systematic presentation of a new approach
to such questions of integral Galois module structure, in a somewhat
generalized sense. This approach is restricted to totally ramified
extensions of local fields $L/K$, whose degree is a power $p^n$ of the
residue characteristic $p$, and which admit an action by a $K$-algebra
$A$ of dimension $p^n$. An {\em $A$-scaffold on $L$} consists of
certain special elements in $A$ which act on suitable elements of $L$
in a way which is tightly linked to the valuation on $L$. The most
obvious setting where scaffolds may occur is that described above,
where $L/K$ is a Galois extension with Galois group $G=\Gal(L/K)$ and
$A=K[G]$. Our approach is not, however, limited to that situation. We
will show in \S5 how it can be applied to a divided power Hopf algebra
$A$ acting in many different ways on an inseparable field
extension. Other inseparable examples have been given by Koch
\cite{Koch-2, Koch-3}. Our approach could also be used for different
Hopf Galois structures on a given separable (but not necessarily
normal) field extension, as described by Greither and Pareigis
\cite{GP}.

When $L/K$ admits an $A$-scaffold, $L$ is a free module over $A$, in
analogy to the Normal Basis Theorem of Galois theory.  We can then
consider any fractional ideal $\euP_L^h$ of $\euO_L$ as a module over
its associated order in $A$,
  $$ \euA(h,A)=\{\alpha\in A :\alpha \cdot \euP_L^h\subseteq
\euP_L^h\}, $$ and ask whether it is a free module. It is in this
sense that our work is concerned with ``generalized'' integral Galois
module structure.  An $A$-scaffold comes with a ``precision''
parameter, and the existence of a scaffold of high enough precision
will enable us to extract a considerable amount of information about
the $\euA(h,A)$-module $\euP_L^h$; not only can we determine if it is
free, but (following \cite{desmit:2} for extensions of degree $p$), we
can also find the minimal number of generators required when it is not
free, and obtain the embedding dimension of $\euA(h,A)$. An important
feature of our approach is that all this information depends on purely
numerical data, namely certain parameters $b_i$ attached to the
scaffold (playing the role of ramification breaks) and the exponent
$h$ of the ideal $\euP_L^h$ under consideration. Given the existence
of a scaffold with specified parameters $b_i$, our results are
therefore, in some sense, universal: they are independent of the
characteristic ($0$ or $p$) of the fields involved, and, in the Galois
case, independent of the precise structure of the Galois group. In
particular, our results make no distinction between abelian and
non-abelian extensions. Moreover, we obtain exactly the same results
for, say, inseparable extensions as for Galois extensions, provided
that the parameters coincide.

The intuition underlying our notion of a scaffold can be explained
somewhat informally as follows. Let $v_K$, $v_L$ denote normalized
valuations such that $v_K(K^\times)=v_L(L^\times)=\mathbb{Z}$. Given
any positive integers $b_i$ for $1\leq i\leq n$ such that $p\nmid
b_i$, there are elements $X_i\in L$ such that $v_L(X_i)=-p^{n-i}b_i$.
Since the valuations, $v_L$, of the monomials
$$\mathbb{X}^a=X_n^{a_{(0)}}X_{n-1}^{a_{(1)}}\cdots X_1^{a_{(n-1)}}:
0\leq a_{(i)}<p,$$ provide a complete set of residues modulo $p^n$
and, since $L/K$ is totally ramified of degree $p^n$, these monomials provide a
convenient $K$-basis for $L$.  The action of $A$ on $L$ is clearly
determined by its action on the monomials $\mathbb{X}^a$.   So if there were
$\Psi_i\in A$ for $1\leq i\leq n$ such that each $\Psi_i$
acts on the monomial basis element $\mathbb{X}^a$ of $L$ as if
it were the differential operator $\partial/\partial X_i$ (with the $X_i$ treated as
independent variables), namely
\begin{equation} \label{deriv-exact}
  \Psi_i\mathbb{X}^a=a_{(n-i)}\mathbb{X}^a/X_i, 
\end{equation}
then the monomials in the $\Psi_i$ (with exponents at most $p-1$)
would furnish a convenient basis for $A$ whose effect on the
$\mathbb{X}^a$ would be easy to follow. As a consequence, the
determination of the associated order of a particular ideal
$\euP_L^h$, and of the structure of this ideal as a module over its
associated order, would be reduced to purely numerical considerations. 
This remains true if (\ref{deriv-exact})
is loosened to the congruence
\begin{equation}  \label{deriv-cong} 
   \Psi_i\mathbb{X}^a\equiv a_{(n-i)}\mathbb{X}^a/X_i 
\pmod{(\mathbb{X}^a/X_i) \euP_L^{\tol}} 
\end{equation}
for a sufficiently large ``precision'' $\tol$. 
The $\Psi_i$, together with the $\mathbb{X}^a$, constitute an
$A$-scaffold on $L$. Our formal definition of an $A$-scaffold
(Definition \ref{tol-scaffold}) is a generalization of this
situation. When the equality (\ref{deriv-exact}) holds, our scaffold
has precision $\tol=\infty$. 

We now explain the background to this work. In the papers 
\cite{elder:scaffold, elder:sharp-crit,byott:scaffold}, the first- and
third-named authors began to develop the theory of scaffolds in the
setting of Galois extensions. There, the parameters $b_i$ of these
Galois scaffolds are just the ramification breaks of the extension
$L/K$. These scaffolds all have precision $\infty$, apart from those on 
cyclic extensions of degree $p^2$ in \cite{elder:sharp-crit}. The main
result of \cite{elder:scaffold} is the existence of a Galois scaffold
for a certain class of arbitrarily large elementary abelian extensions
in characteristic $p$ (the ``near one-dimensional
extensions''). The Galois module structure of the valuation ring in such 
extensions $L/K$ is investigated in \cite{byott:scaffold}, where a necessary
and sufficient condition (in terms of the $b_i$) is given for $\euO_L$
to be free over $\euA_{L/K}$.  This condition turns out to be
equivalent to that given by Miyata \cite{miyata:cyclic2} (and
reformulated in \cite{byott:QJM}) for a class of cyclic Kummer
extensions in characteristic 0. The striking observation that the same
numerical condition holds for two apparently unrelated families of
extensions, differing both in Galois group and in characteristic,
suggests that the methods used in \cite{byott:scaffold} to study
Galois module structure for near one-dimensional extensions might be
applied more widely. The present 
paper develops the machinery to substantiate this idea, while
\cite{elder:sharp-crit} indicates the limitations of our approach by
demonstrating that most extensions will not admit a scaffold. In any
case, our method is necessarily restricted to totally ramified
extensions of $p$-power degree, since if $L/K$ admits an $A$-scaffold,
then it possesses a ``valuation criterion'': there is an integer $b$
such that any element of $L$ of valuation $b$ is a free generator of
$L$ over $A$ (see Proposition \ref{NBT}). This property, which can be
viewed as a strong version of the Normal Basis Theorem, has been
studied in a number of papers \cite{elder:blms, thomas,
  elder:cor-criterion, byott:HVC, desmit:4}, and can only hold when $L/K$ is
totally ramified and of $p$-power degree (see
\cite[Proposition 1.2]{desmit:4} for the Galois case).

When the residue field of $K$ is perfect, we know from
\cite{elder:scaffold} that Galois scaffolds exist for all totally
ramified biquadratic extensions in characteristic $2$, and for all
totally and weakly ramified $p$-extensions in characteristic $p$. To
illustrate the sort of explicit information our methods can yield, we
examine these two classes of extensions in detail (see Theorems
\ref{biquad} and \ref{weak}). However, in this paper we are not
primarily concerned with the problem of actually constructing
$A$-scaffolds. In a separate paper \cite{byott-elder2}, we give a
criterion for a totally ramified Galois $p$-extension to have a Galois
scaffold of a given precision.  This enables us to give an explicit
construction for a class of extensions in characteristic $0$ which
admit Galois scaffolds. These have elementary abelian Galois groups of
arbitrarily large rank, and are the analogs in characteristic $0$ of
the near one-dimensional extensions in characteristic $p$ constructed
in \cite{elder:scaffold}. They include the totally ramified
biquadratic extensions and the totally and weakly ramified
$p$-extensions satisfying some additional hypotheses. Under these
hypotheses, our Galois module results for biquadratic and weakly
ramified extensions in characteristic $p$ will also hold in
characteristic $0$.

Our work is somewhat similar in spirit to that of Bondarko
\cite{bondarko,bondarko:ideals, bondarko:leo}, who considers the existence
of ideals free over their associated orders in the context of totally
ramified extensions of $p$-power degree. (Unlike us, Bondarko only
considers Galois extensions.) Bondarko introduces
the class of semistable extensions. Any such extension contains at
least one ideal free over its associated order, and all such ideals
can be determined from numerical data. Moreover, any abelian extension
containing an ideal free over its associated order, and satisfying
certain additional assumptions, must be semistable. Abelian semistable
extensions can be completely characterized in terms of the Kummer
theory of (one-dimensional) formal groups. 
It would be of interest to understand the precise relationship
between Bondarko's approach and our own, and we
intend to return to this question in future work.

Finally, regarding the hypotheses needed on the ground field $K$, we
note that our main results on $A$-scaffolds do not require the residue
field of $K$ to be perfect. However, in order to construct scaffolds
on particular families of Galois extensions (as we do in
\cite{elder:scaffold, elder:sharp-crit,byott:scaffold, byott-elder2}),
this hypothesis is essential. The hypothesis is also convenient when
discussing higher ramification groups, since the standard exposition
\cite{serre:local} of higher ramification theory makes use of it at
various points. We will therefore not include the condition that $K$
has perfect residue field among the running hypotheses of this paper,
but will impose it from time to time when considering examples.

\subsection{Outline of the paper}

In \S\ref{A-scaffold-sect} we define the notion of an $A$-scaffold on
$L$ and obtain some of its properties. A detailed discussion of the
relationship between the $A$-scaffolds considered here and the Galois
scaffolds of our earlier papers is relegated to an Appendix at the end
of the paper. Our main results, Theorems \ref{IGMS} and
\ref{gencount}, relating $A$-scaffolds to generalized integral
Galois module structure, will be stated and proved in \S3.  In \S4, we
give some applications of our approach to Galois extensions,
discussing in detail biquadratic extensions and weakly ramified
extensions. Finally, in \S\ref{inseparable}, we consider $A$-scaffolds
on inseparable extensions $L/K$, where $A$ is a divided power Hopf
algebra.

\section{$A$-scaffolds}  \label{A-scaffold-sect}

In this section, we consider a totally ramified extension $L/K$ of
local fields, together with a $K$-algebra $A$ which has a $K$-linear
action on $L$. We assume that the residue field $\kappa$ of $K$ has
characteristic $p>0$. The characteristic of $K$ may be either $0$ or
$p$. We do not require $\kappa$ to be perfect. We assume that $L/K$
has degree $p^n$, and that $\dim_K A=p^n$. 

Before giving the definition of an $A$-scaffold on $L$, we require some
notation. We set $\mathbb{S}_{p^n}=\{0,1,\ldots,p^n-1\}$ and
$\mathbb{S}_{p}=\{0,1,\ldots,p-1\}$, and we identify each
$s\in\mathbb{S}_{p^n}$ with its vector of base-$p$ coefficients
$(s)=(s_{(n-1)},\ldots ,s_{(0)})\in\mathbb{S}_p^n$ where
\begin{equation}\label{first-p-adic}
s=\sum_{i=1}^n s_{(n-i)}p^{n-i}.
\end{equation}
This indexing of the base-$p$ digits as $s_{(n-i)}$, where increasing
values of $i$ correspond to decreasing powers of $p$, is natural in
the context of Galois scaffolds, where the $b_i$ are the ramification
breaks (in increasing order), and we need to consider expressions of
the form $\eub(s)$ defined in (\ref{eub}) below.
We will almost always write $s$ in this way. 

We further endow $\mathbb{S}_{p^n}$ with
a partial order that is based upon the usual multi-index partial order
on $\mathbb{S}_p^n$, writing $s\preceq t$ (or $t \succeq s$) if and
only if $s_{(n-i)}\leq t_{(n-i)}$ for $1\leq i\leq n$.  For the
convenience of the reader, we record some facts.
\begin{lemma}\label{preceq}
Let $s,t\in\Spn$ and write $s=\sum_{i=1}^n s_{(n-i)}p^{n-i}$ and
$t=\sum_{i=1}^n t_{(n-i)}p^{n-i}$ where $s_{(n-i)}$, $t_{(n-i)} \in \Sp$.
Then $s \preceq t$ if and only if $s\leq t$ and there are no carries
in the base-$p$ addition of $s$ and $t-s$. Furthermore, the following
are equivalent: 
\begin{enumerate}
\item[(i)] $s_{(n-i)}+t_{(n-i)} \leq p-1$ for $1 \leq i \leq n$;
\item[(ii)] $s \preceq p^n-1-t$;
\item[(iii)] $t \preceq p^n-1-s$; 
\item[(iv)] $s+t\in\Spn$ and $s \preceq s+t$.
\end{enumerate}
\end{lemma}
\begin{proof}
Assume $s \preceq t$. Then clearly $s\leq t$. Let $m=t-s$, and write 
$m=\sum_{i=1}^n m_{(n-i)}p^{n-i}$ with $m_{(n-i)} \in \Sp$. Since 
$0\leq t_{(n-i)}-s_{(n-i)}<p$, we have $m_{(n-i)}=t_{(n-i)}-s_{(n-i)}$. When
we perform the addition $s_{(n-i)}+m_{(n-i)}$ we get $t_{(n-i)}$ with no carries.
On the other hand, assume that $s\leq t$ and there are no carries
in the base-$p$ addition of $s$ and $m=t-s$. As $m\geq 0$ we have  
$m\in\Spn$, so that
$m=\sum_{i=1}^n m_{(n-i)}p^{n-i}$ for some $m_{(n-i)} \in \Sp$. 
Since there are no carries, $m_{(n-i)}+s_{(n-i)}\leq p-1$ for $1\leq i
\leq n$.
Thus $t_{(n-i)}=m_{(n-i)}+s_{(n-i)}$. Therefore $t_{(n-i)}\geq s_{(n-i)}$
for each $i$, so that $s \preceq t$. The equivalence of (i)--(iv) is
then clear. 
\end{proof}

Associated to an $A$-scaffold on $L$ will be a sequence $b_1, \ldots,
b_n$ of integer {\em shift parameters}, which are required to be
relatively prime to $p$. Using these integers, we define a function
$\eub : \mathbb{S}_{p^n} \longrightarrow \mathbb{Z}$ by
\begin{equation}\label{eub}
\eub(s)=\sum_{i=1}^{n}s_{(n-i)}p^{n-i}b_i. 
\end{equation}
We write $r:\bZ\longrightarrow\mathbb{S}_{p^n}$ for the residue function
$r(a)\equiv a\pmod{p^n}$. The coprimality assumption on the $b_i$ ensures that
$r \circ \eub:\mathbb{S}_{p^n}\longrightarrow\mathbb{S}_{p^n}$ is bijective. 
The function $r\circ(-\eub):\Spn\longrightarrow\Spn$,
defined by $r\circ(-\eub)(s)=r(-\eub(s))$, is therefore also 
bijective. We denote its inverse by
$\eua:\Spn\longrightarrow\Spn$.  Abusing notation, we will also write
$\eua(t)$ for $\eua(r(t))$ where $t\in\mathbb{Z}$, and so regard
$\eua$ as a function $\mathbb{Z} \longrightarrow \Spn$. 
\begin{lemma} \label{eua} 
\ \\ 
\noindent (i) $r\circ\eub$ is determined by the residues $b_i\bmod p^i$;

\noindent (ii) if $b_i\equiv b_n\pmod{p^i}$ for all $i$ then
$\eub(s)\equiv b_n s \pmod{p^n}$ for $s\in\mathbb{S}_{p^n}$;

\noindent (iii) if $s,t\in\mathbb{S}_{p^n}$ and $s\preceq t$
then $\eub(s)+ \eub(t-s)=\eub(t)$;

\noindent (iv) $\eub(\eua(t))\equiv -t \pmod{p^n}$ for all $t \in \bZ$;

\noindent (v) $\eua(-\eub(s))=s$ for all $s \in \Spn$.
\end{lemma}
\begin{proof} Clear.\end{proof}

We are now prepared for the definition.
\begin{definition}[$A$-scaffold on $L$] \label{tol-scaffold}
Let $b_1, \ldots, b_n$, $\eub$ and $\eua$ be as above, and let $\tol
\geq 1$. Then an $A$-scaffold on $L$ of precision $\tol$ with
shift parameters $b_1, \ldots, b_n$ consists of 
\begin{enumerate}
\item[(i)] elements $\lambda_t\in L$ for $t\in\bZ$, such that
  $v_L(\lambda_t)=t$ and $\lambda_{t_1}\lambda_{t_2}^{-1}\in K$ whenever
$t_1\equiv t_2 \pmod{p^n}$.
\item[(ii)] elements $\Psi_i\in A$ for $1 \leq i \leq n$, such that $\Psi_i \cdot
  1=0$, and such that, for each $i$ and for each $t\in \mathbb{Z}$,
there exists a unit $u_{i,t}\in\euO_K^\times$ making the following
congruence modulo $\lambda_{t+p^{n-i}b_i}\euP_L^{\tol}$ hold: 
$$\Psi_i \cdot \lambda_t\equiv \begin{cases} 
          \quad u_{i,t} \lambda_{t+p^{n-i}b_i}&
                 \mbox{if } \eua(t)_{(n-i)}\geq 1,\\ 
	\quad 0&\mbox{if } \eua(t)_{(n-i)}=0.
\end{cases} $$
\end{enumerate}
An $A$-scaffold of precision $\infty$ consists of the above data where 
the congruence in (ii) is replaced by equality.
\end{definition}

\begin{remark}   \label{approx}
Condition (ii) in Definition \ref{tol-scaffold} should be interpreted
as saying that the effect of $\Psi_i$ on $\lambda_t$ is approximated
either by a single term or by $0$. The precision $\tol$ determines the
accuracy of this approximation, with a precision of $\infty$
meaning that the ``approximation'' is exact. In more detail, the
approximation works as follows. Since $\Psi_i$ is associated with an
increase of valuation of $p^{n-i}b_i$, we express the effect of
$\Psi_i$ on the basis $\{\lambda_t : 0 \leq t \leq p^n-1\}$ in terms
of the basis $\{\lambda_{p^{n-i}b_i+s} : 0 \leq s \leq p^n-1\}$. Thus
  we have 
$$ \Psi_i \cdot \lambda_t = \sum_{s=0}^{p^n-1} a_{ts}
  \lambda_{p^{n-i}b_i+s}, \qquad a_{ts} \in K. $$ 
Then (ii) says that $a_{ts} \in \pi^{\lceil (t-s+\tol)/p^n\rceil}
\euO_K$ when $t \neq s$, a condition which is independent of $i$, and
each diagonal coefficient $a_{tt}$ is congruent mod $\pi^{\lceil
  \tol/p^n\rceil}$ to either $0$ or a unit of $\euO_K$, according to a
criterion involving $i$ as well as $t$. We observe that the matrix of
exponents $(\lceil t-s+\tol/p^n\rceil)_{1 \leq t,s \leq p^n}$ is
constant on each of the $2p^n-1$ diagonals (from top left to bottom
right) and the main diagonal $t=s$ resides within a band of $p^n$
diagonals where the exponent is $\lceil \tol/p^n\rceil$. How this band
straddles the main diagonal depends on the residue class $\tol \bmod{p^n}$.
\end{remark}

\begin{remark} \label{u-is-1}
In all the examples of $A$-scaffolds known to date, we can take all
the units $u_{i,t}$ in Definition \ref{tol-scaffold}(ii) to be
$1$. Moreover, we can assume $\lambda_{t_1} = \pi^{(t_1-t_2)/p^n}
\lambda_{t_2}$, for some fixed uniformizing parameter $\pi$ of $K$,
whenever $t_1 \equiv t_2 \pmod{p^n}$. The extra generality allowed in
Definition \ref{tol-scaffold} does not significantly add to the
complexity of our arguments, and is included since the flexibility it
provides may be useful in future applications.
\end{remark} 

The reader should keep in mind the following situation.

\begin{definition}[Galois scaffold]  \label{Gal-scaff} 
Suppose that $L/K$ is a Galois extension with Galois group $G$. We
will call a $K[G]$-scaffold on $L$ a {\em Galois scaffold} if the
residue field $\kappa$ is perfect and the shift parameters $b_i$ of
the scaffold are the (lower) ramification breaks $b_1 \leq \ldots \leq
b_n$ of $L/K$, counted with multiplicity in the following sense: we
set $b_i=\max\{j : |G_j|>p^{n-i}\}$ where $G_j=\{ \sigma \in G :
(\sigma-1)\euO_L \subseteq \euP_L^{j+1}\}$ is the $j$th ramification
group. In particular, the existence of a Galois scaffold means that
the ramification breaks $b_i$ are prime to $p$.
\end{definition}

\begin{remark}
In the setting of Definition \ref{Gal-scaff}, $L$ has a subfield $F$
such that $F/K$ is Galois of degree $p$ with ramification break $b_1$.
Moreover, we have $b_i \equiv b_1 \pmod{p}$ for all $i$ by \cite[IV
  \S2 Prop 11]{serre:local}, and $p \nmid b_1$ unless $K$ has
characteristic $0$ and $b_1$ attains its maximal value,
cf.~(\ref{b1-max}) below. Thus the
requirement $p \nmid b_i$ in Definition \ref{Gal-scaff} is 
very mild.
\end{remark}

As explained in the Appendix, the Galois scaffolds considered in
\cite{elder:scaffold, elder:sharp-crit,byott:scaffold} are all Galois
scaffolds in the sense of Definition \ref{Gal-scaff}.

\begin{example}[Galois extensions of degree $p$] \label{deg-p}

We show that a totally ramified Galois extension $L/K$ of degree $p$
admits a Galois scaffold in almost all cases.  There is a unique
ramification break $b_1$, which in characteristic $p$ may be any
positive integer relatively prime to $p$. In characteristic $0$ we
have
\begin{equation} \label{b1-max}
   b_1 \leq p v_K(p)/(p-1), \mbox{ and } p \nmid b_1 \mbox{ unless }
   b_1 = p v_K(p)/(p-1); 
   \end{equation}
see \cite[IV,\S2, Prop.~11 and Ex.~3]{serre:local}.) 

If we exclude the exceptional case $b_1=pv_K(p)/(p-1)$ in characteristic
$0$ then $p \nmid b_1$, and we can obtain a Galois scaffold as
follows. Let $\Psi_1=\sigma-1$, where $\sigma$ is any generator of
$\Gal(L/K)$, let $\pi$ be a uniformizing parameter of $K$, and let
$\rho \in L$ with $v_L(\rho)=b_1$. Then $\eub:\mathbb{S}_p
\longrightarrow \mathbb{Z}$ and $\eua:\mathbb{Z} \to \mathbb{S}_p$ are
given by $\eub(s)=b_1 s$ and $b_1 \eua(t) \equiv -t \pmod{p}$. In
particular, $\eua(b_1)=p-1$. For each $t \in \mathbb{Z}$, put
$f_t=(t-b_1-b_1\eua(b_1-t))/p \in \mathbb{Z}$. Then the elements
$\lambda_t:=\pi^{f_t} \Psi_1^{\eua(b_t-t)} \cdot \rho$
satisfy condition (i) of Definition \ref{tol-scaffold}. Also, $\Psi_1
\cdot 1=0$, and $\Psi_1 \cdot \lambda_t = \lambda_{t+b_1}$ unless
$\eua(b_1-t)=p-1$. But $\eua(b_1-t)=p-1$ precisely when $t \equiv 0
\pmod{p}$, in which case $t=v_L(\lambda_t)=pf_t+pb_1$, $\eua(t)=0$, and
$\Psi_1 \cdot \lambda_t= \pi^{f_t}\Psi_1^p \cdot \rho$. If $K$ has
characteristic $p$ then $\Psi_1^p=(\sigma-1)^p=0$, so $\Psi_1 \cdot
\lambda_t=0$ and we have a Galois scaffold of precision $\tol=\infty$. Now
suppose that $K$ has characteristic $0$. Expanding
$(\Psi_1+1)^p=\sigma^p=1$, we have $\Psi_1^p = - \sum_{j=1}^{p-1}
\binom{p}{j} \Psi_1^j$. Hence
\begin{eqnarray*} 
  v_L(\Psi_1 \cdot \lambda_t) & = & 
  v_L \left( \pi^{f_t} p \Psi_1 \cdot \rho \right) \\
        & = & pf_t + pv_K(p) + b_1 + v_L(\rho) \\
        & = & (t-pb_1) +pv_K(p) + 2b_1. 
\end{eqnarray*}
Thus $v_L(\Psi_1 \cdot \lambda_t) = t+b_1 +[pv_K(p) - (p-1)b_1]$
when $\eua(t)=0$, so we have a Galois scaffold of precision
$\tol=pv_K(p)-(p-1)b_1$.  
\end{example} 

\begin{remark} \label{adjust-shift}
If we replace the element $\Psi_i$ in an $A$-scaffold by $\pi \Psi_i$,
where $\pi$ is some uniformizing parameter of $K$, then we obtain a
new scaffold with the same precision $\tol$, but with the shift
parameter $b_i$ replaced by $b_i 
+p^i$. Suppose that $L/K$ is a Galois extension with ramification breaks
$b_1, \ldots, b_n$. If there exists a Galois scaffold on $L$ (whose
shift parameters are, by definition, the $b_i$), we can adjust the
$\Psi_i$ by powers of $\pi$ to obtain a $K[G]$-scaffold whose shift
parameters are any integers $b'_i$ with $b'_i \equiv b_i \pmod{p^i}$;
this new scaffold will in general not be a
Galois scaffold, since its shift parameters will not coincide with the
ramification breaks. We do not know whether it is possible to have a
$K[G]$-scaffold on a Galois extension $L/K$ with shift parameters
$b'_1, \ldots, b_n'$ that do not satisfy the congruences $b'_i \equiv b_i
\pmod{p^i}$. We do know from \cite{elder:sharp-crit}, however, that if 
$L/K$ is a $C_3\times C_3$-extension in characteristic $3$, and there
exists a $K[G]$-scaffold on $L$ with precision $\tol \geq 1$ and some
shift parameters $b_1'$, $b_2'$, then there will also exist a Galois
scaffold on $L$ (with the ramification breaks $b_1$, $b_2$ as its
shift parameters) of precision $\tol=\infty$.
\end{remark}

\begin{remark} 
In an earlier version of this paper, we called $\tol$ the
``tolerance'' of the scaffold, and this terminology is used by Koch in
\cite{Koch-2}. We thank the referee for suggesting the more
satisfactory word ``precision''.
\end{remark} 

For each $s=\sum_{i=1}^n s_{(n-i)} p^{n-i} \in \Spn$, 
let $\Upsilon^{(s)}$ be the set
of monomials in the (not necessarily commuting) elements $\Psi_ i$
such that, for each $1\leq i\leq n$, the exponents associated with
$\Psi_i$ in the monomial sum to $s_{(n-i)}$.  We write $\Psi^{(s)}$
for the distinguished element
\begin{equation} \label{def-Psi} 
  \Psi^{(s)}=\Psi_n^{s_{(0)}}\Psi_{n-1}^{s_{(1)}}\cdots
    \Psi_1^{s_{(n-1)}} \in \Upsilon^{(s)}. 
\end{equation}
When $A$ is commutative, we have $\Upsilon^{(s)}=\{\Psi^{(s)}\}$.

Suppose that we have an $A$-scaffold as in Definition
\ref{tol-scaffold}. Then it follows inductively that if $t \in \bZ$,
$s \in \Spn$ and $\Psi \in \Upsilon^{(s)}$ then there is a unit
$U_{\Psi,t} \in \euO_K^\times$ such that, modulo
$\lambda_{t+\eub(s)}\euP_L^{\tol}$, we have
\begin{equation}\label{graded}
  \Psi \cdot \lambda_t\equiv \begin{cases}
                \quad U_{\Psi,t}\lambda_{t+\eub(s)}&\mbox{if  
    } s\preceq \eua(t),\\ 
	\quad 0&\mbox{otherwise,}
\end{cases}\end{equation}
and hence
\begin{equation} \label{part-graded-val}
  v_L( \Psi \cdot\lambda_t) \quad \begin{cases} 
   \quad = t+\eub(s) & \mbox{ if } s \preceq \eua(t), \cr
   \quad \geq t + \eub(s) + \tol & \mbox{ otherwise}. 
 \end{cases}
\end{equation}
Thus we have 
\begin{equation}\label{part-graded}
\Psi \cdot \euP_L^t \subseteq \euP_L^{t+\eub(s)}\mbox{ for all } \Psi \in
\Upsilon^{(s)}, \, s\in\Spn,t\in\bZ.
\end{equation}
In particular, \eqref{graded}, \eqref{part-graded-val} and
\eqref{part-graded} hold for 
$\Psi=\Psi^{(s)}$.

\begin{remark}  \label{commute}
Consider the special case of Definition \ref{tol-scaffold} when the
precision is infinite, $\tol=\infty$, and the units are trivial,
$u_{i,t}=1$ for all $i$, $t$. Taking $\Psi=\Psi^{(s)}$ in
(\ref{graded}), we then have the equality
$$ \Psi^{(s)} \cdot \lambda_t = \begin{cases} \quad
  \lambda_{t+\eub(s)}&\mbox{if } s\preceq \eua(t),\\ \quad
  0&\mbox{otherwise.} \end{cases} $$ 
From this we may check that $\{\Psi^{(s)} : s\in\Spn\}$ is a $K$-basis
of $A$ and that $L$ is a free $A$-module of rank $1$ (cf.~Proposition \ref{NBT}
below). Moreover, 
$\Psi^{(r)} \cdot ( \Psi^{(s)} \cdot \lambda_t) = \Psi^{(s)} \cdot (
\Psi^{(r)} \cdot \lambda_t)$ for all $r$, $s \in \Spn$ and all $t \in
\mathbb{Z}$, so that the algebra $A$ is commutative in
this case.  In general, there are two potential sources of
noncommutativity in $A$, namely the ``error terms'' which are implied
by the congruences of Definition \ref{tol-scaffold}(ii), and the units
$u_{i,t}$.

To help fix ideas, we specialize further, assuming in addition that the
shift parameters all satisfy $b_i=1$. (Any totally and weakly ramified
$p$-extension in characteristic $p$ has a scaffold satisfying these
conditions; see \S\ref{weakly} below.)  Then $\eub(s)=s$ for all
$s\in\Spn$, and \eqref{part-graded} states that $$\Psi^{(s)} \cdot
\euP_L^t\subseteq \euP_L^{t+s}\mbox{ for all } 
s\in\Spn,t\in\bZ.$$ 
\end{remark} 

The Normal Basis Theorem ensures, in the Galois case, that $L$ is a
free $K[G]$-module of rank $1$. We now show that a similar assertion
holds whenever we have an $A$-scaffold. Furthermore, $L/K$ satisfies
the stronger condition of having a ``valuation criterion'' for its
$A$-module generator.
 
\begin{proposition} \label{NBT}
Let $L/K$ have an $A$-scaffold of precision $\tol \geq 1$. Then
$\{\Psi^{(s)} :s\in\Spn\}$ is a $K$-basis of $A$. Moreover, let $b$
be any integer that satisfies $\eua(b) =p^n-1$, and let $\rho\in L$
with $v_L(\rho)=b$.  Then $L$ is a free $A$-module on the generator
$\rho$. Additionally, for each $h \in \bZ$, the ring
$\euA(h,A)=\{\alpha\in A:\alpha \cdot \euP_L^h\subseteq \euP_L^h\}$ is
an $\euO_K$-order in $A$.
\end{proposition}
\begin{proof}
Since $\eua \colon \Spn \to \Spn$ is bijective, the condition
$\eua(b) =p^n-1$ determines $b$ uniquely mod $p^n$. We have 
$\rho=u\lambda_{b}+\sum_{i>b}a_i\lambda_i$ for $u\in\euO_K^\times$
and $a_i\in\euO_K$.  From (\ref{part-graded-val}), for $i>b$ and for
each $s \in \mathbb{S}_{p^n}$ we have 
$v_L(\Psi^{(s)} \cdot a_i\lambda_i)>v_L(\Psi^{(s)} \cdot
u\lambda_{b})=b+\eub(s)$. Thus
$v_L(\Psi^{(s)} \cdot \rho)=b+\eub(s)$ for each $s \in
\mathbb{S}_{p^n}$. Since $\eub : \mathbb{S}_{p^n} \longrightarrow
\mathbb{S}_{p^n}$ is surjective, these valuations represent all
residue classes mod $p^n$. As $L/K$ is totally ramified, it follows
that $\{\Psi^{(s)} \cdot \rho:s\in\Spn\}$ is a $K$-basis for
$L$. Thus $A \cdot \rho = L$, and, comparing dimensions, $L$
is a free $A$-module on the generator $\rho$. Moreover, the
$\Psi^{(s)}$ must be linearly independent over $K$. Since $\dim_K
A=p^n$, it follows that the $\Psi^{(s)}$ form a $K$-basis of $A$. As
$L$ is a free $A$-module and $\euP_L^h$ spans $L$ over $K$, it is
immediate that $\euA(h,A)$ is an $\euO_K$-order in $A$.
\end{proof}

\begin{remark}\label{galois-case}
Suppose we have a Galois scaffold on an abelian extension $L/K$.  By
the Hasse-Arf Theorem \cite[V, \S7]{serre:local}, the ramification
breaks $u_1, \ldots, u_n$ in the upper numbering are
integers. Translating to the lower numbering, we obtain the
congruences $b_i \equiv b_n \pmod{p^i}$. Thus we have $\eub(s) \equiv b_n
s \pmod{p^n}$ and $b_n \eua(t) \equiv -t \pmod{p^n}$. In particular, we
can then take $b$ in Proposition \ref{NBT} to be $b_n$. The same will
hold if $L/K$ is a nonabelian Galois extension which satisfies the
conclusion of the Hasse-Arf Theorem.

If $L/K$ is a Galois extension not necessarily satisfying the
conclusion of the Hasse-Arf Theorem, then the $u_i$ need not be
integers. In this case, the condition $\eua(b)=p^n-1$ is equivalent to
$b \equiv b_n-p^n u_n \pmod{p^n}$. Thus Proposition \ref{NBT} agrees
with the valuation criterion for a normal basis generator in
\cite{elder:cor-criterion}.
\end{remark}

\section{Integral $A$-module structure} \label{IGMS-sect}

\subsection{Statement of the main results}
Fix $L/K$ and $A$ as in \S\ref{A-scaffold-sect}.  Assume that there is
an $A$-scaffold on $L$ of precision $\tol\geq 1$ as in Definition
\ref{tol-scaffold}.  Thus we have shift parameters $b_1, \ldots, b_n$ and
the associated functions $\eub$ and $\eua$, as well as elements
$\lambda_t \in L$ with $v_L(\lambda_t)=t$ for each $t \in \bZ$. 
By Proposition \ref{NBT}, we also have a $K$-basis 
$\{\Psi^{(s)}:s\in\Spn\}$ of $A$. We
choose once and for all a uniformizing parameter $\pi$ of $K$. 

Now let $h \in \bZ$, and consider the
fractional $\euO_L$-ideal $\euP_L^h$ as a module over its associated
order 
\begin{equation} \label{def-Ah}
 \euA:=\euA(h,A)=\{\alpha\in A:\alpha \cdot \euP_L^h\subseteq
\euP_L^h\}
\end{equation}
in $A$. If $h'=h+p^n m$ for some $m \in \bZ$ then $\euP_L^{h'}= \pi^m
\euP_L^h$. It follows that
$\euA(h',A)=\euA(h,A)$, and that $\euP_L^{h'}$ and $\euP_L^h$ are
isomorphic as modules over this order. Thus $h$ only matters up to
congruence mod $p^n$.

Both the order $\euA$, and the structure of $\euP_L^h$ over
$\euA$, depend only on the residue class $h \pmod{p^n}$.  Let
$\Spn(h)=\{t \in \bZ \ :\ h \leq t <h+p^n\}$. Note that
$\Spn(0)=\Spn$, and that $\{\lambda_t \: :\ t \in \Spn(h)\}$ is an
$\euO_K$-basis for $\euP_L^h$. We now fix a specific choice of $b$ in
Proposition \ref{NBT} (where $b$ was only determined mod $p^n$) by
stipulating
\begin{equation} \label{b-def}
 \eua(b)=p^n-1, \qquad  b  \in \Spn(h).  
 \end{equation}
Thus we have  $L=A\cdot\lambda_b$.

For each $s \in \Spn$ we define
\begin{equation} \label{def-d}
   d(s)= \left\lfloor \frac{\eub(s)+b-h}{p^n}\right\rfloor. 
\end{equation}
In particular, $d(0)=0$ since $b-h \in \Spn$. 
We also define 
\begin{equation} \label{def-ws}
 w(s) = \min \{ d(u)-d(u-s)
 \ : \ u \in \Spn, \; u \succeq s\}.
\end{equation}
Using Lemma \ref{preceq}, we have
$$ w(s) = \min \{ d(s+j)-d(j) \ : \ j \in \Spn, \; j \preceq
p^n-1-s\}.  $$ 
In particular, $d(s)-1 \leq w(s) \leq d(s)-d(0)=d(s)$ for all 
$s\in\Spn$. Note that whether or not
the upper bound $w(s)=d(s)$ is achieved depends only on the residue
classes $b_i \pmod{p^i}$, not the integers $b_i$ themselves.  In any
case, it is important to realize that both $d(s)$ and $w(s)$, as well
as $b$ and $b-h$, depend on $b_1, \ldots, 
b_n$ and on $h$, although we do not indicate this dependence explicitly in our
notation. 

For $s \in \Spn$, we normalize the $\Psi^{(s)}$ in \eqref{def-Psi},
and set
$$ \Phi^{(s)} = \pi^{-w(s)} \Psi^{(s)} .$$

The first of our main results explains how the existence of an
$A$-scaffold of high enough precision allows us to give an explicit 
description of $\euA$, and to determine whether or not 
$\euP_L^h$ is free over $\euA$, using only the numerical invariants
$w(s)$ and $d(s)$.

\begin{theorem}\label{IGMS}
Let $L/K$ admit an $A$-scaffold of precision $\tol$ with shift
parameters $b_1,\ldots ,b_n$.  Fix a fractional ideal $\euP_L^h$, and
let $\euA$, $b$, $d(s)$ and $w(s)$ be defined as in
(\ref{def-Ah})--(\ref{def-ws}).
\begin{itemize}
\item[(i)] 
Suppose that $\tol \geq \max(b-h,1)$. Then   $\{\Phi^{(s)}: s \in
\Spn\}$ is an $\euO_K$-basis of $\euA$. If 
$w(s)=d(s)$ for all $s \in \Spn$, then 
$\euP_L^h$ is free over $\euA$ with $\euP_L^h=\euA \cdot \lambda_b$.
\item[(ii)]
Now suppose that the stronger condition $\tol\geq p^n+b-h$ holds. Then 
$\euP_L^h$ is free over $\euA$ if and only if $w(s)=d(s)$
  for all $s \in \Spn$.  Moreover, when $\euP_L^h$ is free over
  $\euA$, we have $\euP_L^h=\euA \cdot \rho$ for any $\rho\in L$ with
  $v_L(\rho)=b$. 
\end{itemize}
\end{theorem}

\begin{remark} \label{2pn}
Since $b$ was chosen so that $b-h\in \Spn$, the
stronger condition $\tol\geq p^n+b-h$ holds for all ideals
if the $A$-scaffold has precision $\tol \geq 2p^n-1$.
\end{remark}

\begin{example}[Galois extensions of degree $p$] \label{deg-p-consequences}
For a totally ramified Galois extension $L/K$ of degree $p$, the
Galois module structure, both of the valuation ring $\euO_L$ and of its
fractional ideals $\euP_L^h$, has been studied extensively.  We
briefly review the existing results and relate them to Theorem
\ref{IGMS}.

For the valuation ring itself, we have $h=0$, so the number $b$ in
Theorem \ref{IGMS} is just the least positive residue $r(b_1)$ of
$b_1$ mod $p$. For $K$ of characteristic $0$, Bertrandias and Ferton
\cite{ferton} show that $\euO_L$ is free over its associated order
if and only if $b$ divides $p-1$, provided that $b_1$ is not too close
to its maximal value. (See \cite{bbferton} for the excluded cases.)
Now $d(s)=\lfloor (b_1 s+b)/p \rfloor$, and one can verify that our
condition $w(s)=d(s)$ in this case is equivalent to $b \mid (p-1)$. We
therefore recover the result of Bertandias and Ferton whenever we
have a Galois scaffold with $\tol \geq b+p$; by Example \ref{deg-p}, this
occurs when
\begin{equation} \label{b1-deg-p}
   b_1 < \frac{p v_K(p)}{p-1} - 2.
\end{equation}
In characteristic $p$, Aiba \cite{aiba} gives a different condition for
$\euO_L$ to be free, but his condition can be shown to be equivalent
to $b \mid (p-1)$; de Smit and Thomas \cite{desmit:2} also obtain $b
\mid (p-1)$. Since there is a Galois scaffold with $\tol=\infty$,
these results follow from our Theorem \ref{IGMS}, exactly as in
characteristic $0$ (but with no upper bound on $b_1$). 

We now consider arbitrary ideals $\euP_L^h$. In characteristic $0$,
Ferton \cite{ferton:ideals} determines which ideals are free over
their associated orders, giving her result in terms of the continued
fraction expansion of $b_1/p$.  A corresponding result in
characteristic $p$ is given by Huynh \cite{huynh}, who gives a
different criterion but proves it is equivalent to Ferton's. Our
condition, $w(s)=d(s)$ for all $s$, must therefore be equivalent to
Ferton's continued fraction criterion. This equivalence is verified in
\cite{Maria} (which also contains some partial results relating our
Theorem \ref{gencount} below to continued fractions). Given this
equivalence, and assuming (\ref{b1-deg-p}) in the characteristic $0$
case, the results of Ferton and Huynh follow from our Theorem
\ref{IGMS}.
\end{example}

The following example considers another situation where the technical
details associated with Theorem \ref{IGMS} are easy to digest.

\begin{example}[$b_i \equiv -1$]\label{bi-1}
Suppose that $L/K$ is totally ramified extension of degree $p^n$ (for
arbitrary $n \geq 1$) which admits an $A$-scaffold with precision
$\tol \geq p^n-1$ such that $b_i \equiv -1 \pmod{p^i}$ for each
$i$. We consider the valuation ring $\euO_L$ (so $h=0$). Write
$b_i=-1+m_i p^i$ with $m_i \in \mathbb{Z}$. Using \eqref{eub}, we see
that $\eub(s)=-s+(\sum_{i=1}^ns_{(n-i)}m_i)p^n\equiv -s
\pmod{p^n}$. Thus $b=p^n-1$ and $d(s)=\sum_{i=1}^ns_{(n-i)}m_i =
\sum_{i=1}^ns_{(n-i)} d(p^{n-i})$.  In particular, $d(s)+d(j)=d(s+j)$
for all $j\in\Spn$ with $j\preceq p^n-1-s$, so that $w(s)=d(s)$ for
all $s$. Moreover, $w(s)=\sum_{i=0}^{n-1}s_{(i)}w(p^{i})$. Thus by
Theorem \ref{IGMS}(i), $\euO_L$ is free over $\euA$, and $\euA$ has
the particularly simple form:
$$ \euA  = \euO_K\left[\pi^{-m_1} \Psi_1,\pi^{-m_2} \Psi_2,\ldots,
  \pi^{-m_n} \Psi_n \right]. $$
\end{example}

We make one further remark, concerning the precision in
Theorem \ref{IGMS}. 

\begin{remark} \label{tol-or-alg}
In some cases it is possible to relax the assumptions on $\tol$ in
Theorem \ref{IGMS} at the expense of stronger assumptions on the
$\Psi_i$ in Definition \ref{tol-scaffold}. For example, in
\cite[Theorem 1.1]{elder:sharp-crit} we give a freeness criterion,
which is equivalent to that in Theorem \ref{IGMS}(ii), for the
valuation ring of a cyclic extension of degree $p^2$ in characteristic
$p$ admitting a different sort of ``scaffold''.  From the
perspective of Definition \ref{tol-scaffold}, this is a Galois
scaffold of precision $\tol=b_2-pb_1$, but this value is not used in
the proof of the result. In fact, although the residue class
$b_1\equiv b_2\pmod{p^2}$ satisfied by the ramification breaks could be
any class mod $p^2$ relatively prime to $p$, the proof of the result
requires only that the ``scaffold'' have precision $\tol \geq 1$. In
contrast, we would need to assume that $\tol \geq 2p^2-1$ to guarantee
that Theorem \ref{IGMS} applies for all possible values of the
ramification breaks. The result in \cite{elder:sharp-crit} depends on
the fact that the ``scaffold'' there satisfies the additional
relations $\Psi_1^p=\Psi_2$ and $\Psi_2^p=0$.
\end{remark}

The second of our main results, Theorem \ref{gencount}, adapts the
techniques of \cite{desmit:2} (see in particular Theorem 4) to extract
some further information from the numerical data $d(s)$ and
$w(s)$. For $s$, $t \in \Spn$, we write $s \prec t$ if $s \preceq t$
and $s \neq t$. Let
$$ \DD = \{ u \in \Spn \, : \, d(u)>d(u-s)+w(s) \mbox{ for all } s \in
\Spn \mbox{ with } 0 \prec s \preceq u \}; $$
$$ \EE = \{ u \in \Spn \, : \, w(u)>w(u-s)+w(s) \mbox{ for all } s \in
\Spn \mbox{ with } 0 \prec s \prec u \}. $$ 
Note that $0 \in \DD$ and $0$, $1, p, \ldots, p^{n-1} \in \EE$ since
there are no relevant $s$ in these cases. Thus we always have $|\DD|
\geq 1$ and $|\EE| \geq n+1$. Again, the dependence on $h$ and on the
$b_i$ is suppressed from the notation.

\begin{theorem} \label{gencount}
Let $L/K$ be as in Theorem \ref{IGMS}, with the strong condition $\tol\geq
p^n+b-h$. Then the minimal number of generators of the $\euA$-module
$\euP_L^h$ is $|\DD|$.  Also, $\euA$ is a (not necessarily
commutative) local ring with residue field $\kappa = \euO_K/\euP_K$,
and, writing $\euM$ for its unique maximal ideal, the embedding
dimension $\dim_\kappa (\euM/\euM^2)$ of $\euA$ is $|\EE|$.
\end{theorem}
 
Since $L$ is a free $A$-module by Proposition \ref{NBT}, the minimal
number of generators of $\euP_L^h$ over $\euA$ is one precisely when
$\euP_L^h$ is free over $\euA$.

\subsection{Proofs}  \label{IGMS-proof}

We keep the notation of the previous subsection. In particular, $L/K$
admits an $A$-scaffold with precision $\tol \geq 1$ and with shift
parameters $b_1. \ldots, b_n$, giving rise to the functions $\eub :
\Spn \to \mathbb{Z}$ and $\eua : \Spn \to \Spn$. We fix $h \in
\mathbb{Z}$ and study the ideal $\euP_L^h$ as a module over its
associated order $\euA:=\euA(h,a)$. Recall that $b$ is the unique
integer satisfying (\ref{b-def}).

Our goal in this subsection is to prove Theorems \ref{IGMS} and
\ref{gencount}, but we first provide an overview of the strategy of
the proofs. The reader might find it helpful initially to consider the
special case $\tol=\infty$, $u_{i,t}=1$ in Remark \ref{commute} (which
forces $A$ to be commutative), and further to suppose that $b_1=\cdots
=b_n=b$, so that $\eub(s)=bs$.

Let $t \in \Spn(h)$ and $s \in \Spn$. If $s \preceq \eua(t)$ and $\Psi
\in \Upsilon^{(s)}$ then by (\ref{part-graded-val}) the element $\Psi \cdot
  \lambda_t$ has valuation $t+\eub(s)$. We wish to relate this
  element to the $\euO_K$-basis $\{ \lambda_m : m \in \Spn(h)\}$ of
    $\euP_L^h$, so, for any $t \in \Spn(h)$ and $s \in \Spn$, we write
\begin{equation} \label{H-d}
   t+\eub(s) = H(s,t) + p^n D(s,t) \mbox{ with } H(s,t) \in \Spn(h). 
\end{equation}
Thus we have
\begin{equation} \label{def-H-d}
  D(s,t)= \left\lfloor \frac{t+\eub(s)-h}{p^n}\right\rfloor, \qquad
   H(s,t) = h + r(t+\eub(s)-h).
\end{equation}	
In particular, comparing with (\ref{def-d}), we have
\begin{equation} \label{Dd}
  D(s,b) = d(s).
\end{equation}

By Proposition \ref{NBT}, $\lambda_b$ has the normal basis property
$L=A \cdot \lambda_b$, so we seek to compare $\Psi^{(s)} \cdot
\lambda_t$ with $\Psi^{(u)} \cdot \lambda_b$ where $u \in \Spn$ is
chosen to make the valuations of these elements agree mod $p^n$. Thus
we require $H(u,b)=H(s,t)$. There will be a unique $u$ with this
property, since $H(u,b)$ realizes each element of $\Spn(h)$ exactly
once as $u$ varies in $\Spn$.

In order to translate between $t$ and $u$ (for a fixed $s$), we will
need a number of facts which depend on the properties of $\eub$ and
$\eua$ given in Lemma \ref{eua}. These facts are recorded in Lemma
\ref{t-u}. We are interested in the valuations of the
elements $\Phi^{(s)} \cdot \lambda_t = \pi^{-w(s)} \Psi^{(s)} \cdot
\lambda_t$ or, more generally,
$\pi^{-w(s)} \Psi \cdot \lambda_t$ for any $\Psi \in \Upsilon^{(s)}$.
In Proposition \ref{psi-vals} we determine some of these valuations
precisely, and bound the rest in terms of $\tol$. To prove
Theorem \ref{IGMS}, we then use this information to obtain an explicit
description of the associated order $\euA$ and to determine when
$\euP_L^h$ is free over $\euA$.

Before proving Theorem \ref{gencount}, we need to deal with the fact
that $\euA$ need not in general be commutative. We show in
Proposition \ref{psi-prod} that any two of our basis elements
$\Phi^{(r)}$, $\Phi^{(s)}$ of $\euA$ commute mod $\pi \euA$ up to
multiplication by a unit in $\euO_K$.

We begin the proof of Theorem \ref{gencount} by showing that the
$\euO_K$-lattice $\euM$ in $\euA$, spanned
by $\pi$ and the $\Phi^{(s)}$ for $s \neq 0$,  is the unique maximal
ideal of $\euA$. Since $\Psi_i \cdot 1 
=0$, it is easy to see that $\euM$ is an ideal of $\euA$, and that
$\euA/\euM \cong \kappa$, the residue field of $K$. To show the
uniqueness, we check that $\euM$ is topologically nilpotent. This is
easy to see in the special case considered in Remark \ref{commute},
where $A$ is commutative and $\Psi_i^p=0$ for each $i$. In
general, we use Proposition \ref{psi-prod} to show that $\euM$ is
topologically nilpotent.

Once we have established that $\euM$ is the unique maximal ideal of
$\euA$ (so that $\euA$ is a local ring), it follows by Nakayama's
Lemma that the minimal number of generators for the $\euA$-module
$\euP_L^h$ (resp.~$\euM)$ is just the dimension of $\euP_L^h/ \euM
\cdot \euP_L^h$ (resp.~$\euM/\euM^2$) as a vector space over $\kappa$. To
determine these dimensions, we take the obvious $\euO_K$-basis of
$\euP_L^h$ (resp.~$\euM$), which is indexed by the partially ordered
set $\Spn$.  Some of these generators are redundant because they can
be obtained by the action of $\euA$ on another generator occurring
earlier in the partial order. Removing these redundant generators will
leave a basis of the appropriate $\kappa$-vector space, since, by
hypothesis, the precision of the scaffold is too high to allow any
further relations between the surviving generators.

This concludes our overview of the proofs, and we now start the
detailed arguments.

\begin{lemma} \label{t-u}
Fix $s \in \Spn$, and let $t \in \Spn(h)$ and $u \in \Spn$ satisfy
$H(u,b)=H(s,t)$. Then we have 
$$ s \preceq \eua(t) \LRA s \preceq u, $$
Moreover, when $s \preceq \eua(t)$, the following hold:
\begin{itemize}
\item[(i)] $\eua(H(s,t)) = \eua(t) -s$;
\item[(ii)] $u=p^n-1+s-\eua(t)$;
\item[(iii)] $t=H(u-s,b)$;
\item[(iv)] $D(s,t)=d(u)-d(u-s)$.
\end{itemize} 
\end{lemma}
\begin{proof}
Let $s \preceq \eua(t)$. By Lemma
\ref{eua}, we have $\eub(s)+\eub(\eua(t)-s)=\eub(\eua(t))\equiv -t
\pmod{p^n}$. Using (\ref{def-H-d}), it follows that $H(s,t) \equiv
t+\eub(s) \equiv -\eub(\eua(t)-s) \pmod{p^n}$. Applying $\eua$ gives
(i). Similarly, as 
$u \preceq \eua(b)=p^n-1$, we have $H(u,b) \equiv b+ \eub(u) \equiv
\eub(u) - \eub(p^n-1) = -\eub(p^n-1-u)$. Since $H(s,t)=H(u,b)$, we
therefore have $\eua(t)-s=p^n-1-u$,
giving (ii), and thus
$\eua(t)_{(n-i)} - s_{(n-i)} = (p-1)-u_{(n-i)}$ for $1 \leq i \leq
n$, since $s \preceq \eua(t)$ by hypothesis. Hence $s_{(n-i)} =
u_{(n-i)} - (p-1-\eua(t)_{(n-i)})$ for 
each $i$, so that $s \preceq u$. This shows the
implication  $s \preceq \eua(t) \Rightarrow s \preceq u$. The reverse
implication follows since the sets
$\{t \in \Spn(h) : s \preceq \eua(t)\}$ and $\{u \in \Spn : s
\preceq u\}$ have the same cardinality.

It remains to prove (iii) and (iv). Still assuming $s \preceq
\eua(t)$, we have from (\ref{def-H-d}) that 
$$ H(u-s,b) \equiv b + \eub(u-s)  \equiv b + \eub(p^n-1-\eua(t)) \equiv
b + \eub(p^n-1)+t \equiv t \pmod{p^n}, $$
and (iii) follows as both sides are in $\Spn(h)$. Finally, using
(\ref{H-d}) and (\ref{Dd}), we have 
\begin{eqnarray*}
 p^n D(s,t) & = & t + \eub(s) - H(s,t) \\
           & = & H(u-s,b) + \eub(s) - H(u,b) \\
           & = & [b + \eub(u-s) - p^n D(u-s,b)] + \eub(s) - 
                 [b + \eub(u) - p^n D(u,b))] \\
           & = & p^n d(u) - p^n d(u-s),
\end{eqnarray*}
since $\eub(u-s)=\eub(u)-\eub(s)$ because $s \preceq u$. Dividing by
$p^n$ yields (iv).
\end{proof}

It is immediate from Lemma \ref{t-u} that we may rewrite
(\ref{def-ws}) as
\begin{equation} \label{w-bis}
  w(s) = \min\{ D(s,t)  : t \in \Spn(h),\, \eua(t)  \succeq s\} .
\end{equation}
Moreover, it then follows from (\ref{def-H-d}) that if  $s \preceq
\eua(t)$ then either $D(s,t)=w(s)$ or $D(s,t)=w(s)+1$. We define 
\begin{equation} \label{def-eps}
   \epsilon(s,t) = D(s,t)-w(s) \in \{0,1\} \mbox{ for } s \preceq \eua(t). 
\end{equation}

\begin{proposition}   \label{psi-vals}
Suppose that the $\Psi_i$ are as in Definition \ref{tol-scaffold}. Let
$s \in \Spn$ and $t \in \Spn(h)$. Let $\Psi$ be any 
element of $\Upsilon^{(s)}$, and set $\Phi=\pi^{-w(s)} \Psi$. 
\begin{itemize}
\item[(i)] If $s \preceq \eua(t)$ then there is a unit $y_{\Phi,t} \in
  \euO_K^\times$ such that 
$$ \Phi \cdot \lambda_t \equiv \pi^{\epsilon(s,t)} y_{\Phi,t}
\lambda_{H(s,t)} \pmod{\pi^{\epsilon(s,t)}\lambda_{H(s,t)}
\euP_L^\tol}.$$
In particular, 
$$ v_L(\Phi \cdot \lambda_t)= \begin{cases} 
 H(s,t) & \mbox{ if } s \preceq \eua(t) 
    \mbox{ and }\epsilon(s,t)=0,  \\
 H(s,t)+p^n & \mbox{ if } s \preceq \eua(t) 
         \mbox{ and }\epsilon(s,t)=1. \end{cases}  $$
\item[(ii)] If $s \not \preceq \eua(t)$ then we have the bounds
$$ v_L(\Phi \cdot \lambda_t) \geq \begin{cases} 
  H(s,b) + t - b + \tol & \mbox{ if } s \not \preceq \eua(t) \mbox{ and }
  w(s)=d(s), \\
  H(s,b)+t-b +p^n + \tol & \mbox{ if } s \not \preceq \eua(t) \mbox{ and }
  w(s) \neq d(s)
  \end{cases} $$
\end{itemize}
\end{proposition}
\begin{proof}
It follows from \eqref{H-d} and Definition \ref{tol-scaffold}(i)
that there is an $x \in \euO_K^\times$ so that $ \lambda_{t+\eub(s)} =
x \pi^{D(s,t)} \lambda_{H(s,t)}$.  

\noindent (i) If $s \preceq \eua(t)$ then
\eqref{graded} gives $ \Psi \cdot \lambda_t \equiv U_{\Psi,t} \lambda
_{t+\eub(s)} \pmod{\lambda_{t+\eub(s)} \euP_L^\tol}$. Multiplying by
$\pi^{-w(s)}$ and setting $y_{\Phi,t}=xU_{\Psi,t}$ we obtain the
required congruence. The remaining assertions follow
immediately. 

\noindent
(ii) If $s \not \preceq \eua(t)$ then \eqref{graded} gives
$v_L(\Phi \cdot \lambda_t) \geq   t+\eub(s) -p^n w(s) + \tol$. 
From (\ref{H-d}) we have 
$$ t + \eub(s) = H(s,t)+p^n D(s,t) = t-b + H(s,b) + p^nD(s,b). $$
Hence, using (\ref{Dd}),  
$$ v_L(\Phi \cdot \lambda_t) \geq t-b+H(s,b) + p^n(d(s)-w(s)) 
  +\tol, $$
and by (\ref{def-d}) and (\ref{def-ws}) either $w(s) = d(s)$ or
$w(s)=d(s)-1$. The two cases give the stated inequalities.  
\end{proof}

We can now prove the first of our main results,
\medskip

\begin{pf-IGMS}

(i) Assume that $\tol\geq \max(b-h,1)$. By Proposition
\ref{psi-vals}, we have for all $s \in \Spn$ and all $t \in \Spn(h)$
that $v_L(\Phi^{(s)} \cdot \lambda_t) \geq h$. Since $\{\lambda_t: t
\in\Spn(h)\}$ is an $\euO_K$-basis of $\euP_L^h$, this shows that
$\Phi^{(s)} \in \euA$ for all $s$. Any $\alpha \in A$ may be written $
\alpha = \sum_{s \in \Spn} c_s \Phi^{(s)}$ for some $c_s \in K$.  We
have just shown that if $c_s \in \euO_K$ for all $s$ then $\alpha \in
\euA$. We must show, conversely, that if $\alpha \in \euA$ then each
$c_s \in \euO_K$. Applying $\alpha$ to $\lambda_b$, we obtain $\alpha
\cdot \lambda_b = \sum_{s} c_s \Phi^{(s)} \cdot \lambda_b$.  But $s
\preceq \eua(b)=p^n-1$, so, for each $s$ with $c_s \neq 0$, we have
$v_L(c_s \Phi^{(s)} \cdot \lambda_b) \equiv H(s,b) \pmod{p^n}$ by
Proposition \ref{psi-vals}(i). These valuations are distinct mod
$p^n$, so $v_L(c_s \Phi^{(s)} \cdot \lambda_b) \geq h$. Thus
$c_s \in \euO_K$ if $\epsilon(s,b)=0$, and $c_s \in \pi^{-1} \euO_K$
otherwise.  Now assume for a contradiction that some $c_s \not \in
\euO_K$. Since $\epsilon(s,b)=1$, we have $d(s)=D(s,b)=w(s)+1$. By
(\ref{w-bis}), there is some $t \in \Spn(h)$ with $\eua(t)
\succeq s$ and $D(s,t)=w(s)$, so that $\epsilon(s,t)=0$. Amongst these
$t$, take the one with $H(s,t)$ minimal, and consider $\alpha \cdot
\lambda_{t} = \sum_{j\in \Spn} c_j \Phi^{(j)} \cdot \lambda_t$. For
the term $j=s$ we have 
$$ v_L(c_j \Phi^{(j)} \cdot \lambda_t)=v_L(c_s) +
     H(s,t) = -p^n+H(s,t)<h  $$ 
by Proposition \ref{psi-vals}(i). For the terms with $j \neq s$ but $j
\preceq \eua(t)$, we have 
$$ v_L(c_j \Phi^{(j)} \cdot \lambda_t) > -p^n+H(s,t)  $$ 
by Proposition \ref{psi-vals}(i) again and the choice of
$t$. For the terms with $j \not \preceq \eua(t)$, since $w(s)\neq
d(s)$, we have 
$$   v_L(c_j \Phi^{(j)} \cdot \lambda_t) \geq v_L(c_j) +
    H(j,b)+ t-b +p^n + \tol \geq h $$ 
by Proposition \ref{psi-vals}(ii) and the hypothesis on $\tol$. Hence
$$ v_L(\alpha \cdot \lambda_t) =  -p^n+H(s,t) <h,  $$ 
giving the required contradiction.

(ii) Now assume that the stronger condition $\tol \geq p^n+b-h$
holds. Let $\rho$ 
be an arbitrary element of $\euP_L^h$. We investigate when $\rho$ is a
free generator for $\euP_L^h$ over $\euA$.  Since $\{\lambda_t: t \in
\Spn(h)\}$ is an $\euO_K$-basis for $\euP_L^h$, we have $\rho =
\sum_{t\in\Spn(h)} x_t \lambda_t$ for some $x_t \in \euO_K$. By
Proposition \ref{psi-vals} and the hypothesis on $\tol$, we therefore
have
$$   \Phi^{(s)} \cdot \rho  \equiv \sum_t  x_t y_{s,t} \pi^{\epsilon(s,t)}
 \lambda_{H(s,t)} \pmod{\pi \euP_L^h},  $$ 
 where the sum is over those $t \in \Spn(h)$ with $s \preceq \eua(t)$. 
Using Lemma \ref{t-u}, we can rewrite this as
$$ \Phi^{(s)} \cdot \rho \equiv \sum_{u \succeq s} c_{s,u}
 \lambda_{H(u,b)} \pmod{\pi \euP_L^h}, $$ 
where the sum is over $u \in \Spn$ satisfying $u \succeq s$, and where
$c_{s,u} = x_t y_{s,t}\pi^{\epsilon(s,t)}$ for $t=H(u-s,b)$.  The
matrix $(c_{s,u})$ expressing the elements $\Phi^{(s)} \cdot \rho$
(ordered by increasing $s$) in terms of the basis elements
$\lambda_{H(u,b)}$ (ordered by increasing $u$) is therefore upper
triangular mod $\pi$. Thus the $\Phi^{(s)} \cdot \rho$ also form an
$\euO_K$-basis of $\euP_L^h$ if and only if $c_{s,s} \in
\euO_K^\times$ for all $s$. But when $u=s$, we have $t=H(0,b)=b$ and
$D(s,t)=d(s)$.  Since $x_b \in \euO_K$, $y_{s,b} \in \euO_K^\times$,
and $d(s) \geq w(s)$, it follows that $\euP_L^h = \euA \cdot\rho$ if
and only if $x_b \in \euO_K^\times$ and $d(s)=w(s)$ for all $s$. Thus
$\euP_L^h$ is a free $\euA$-module on some generator $\rho$ if and
only if $d(s)= w(s)$ for all $s$. Moreover, if $v_L(\rho)=b$ then we
must have $v_L(x_t \lambda_t) \geq b$ for all $t$, with equality for
$t=b$. In particular, $x_b \in \euO_K^\times$. Hence $\rho$ is a free
generator for $\euP_L^h$ over $\euA$, provided that $d(s)=w(s)$ for
all $s$.
\end{pf-IGMS}

\begin{proposition} \label{psi-prod}
Suppose that $\tol \geq p^n+b-h$ and let $r$, $s \in \Spn$.
If $r \not \preceq p^n-1-s$ or if $w(r)+w(s) \neq w(r+s)$ then 
$\Phi^{(r)}\Phi^{(s)}\in\pi \euA$. In the remaining case 
that $r \preceq p^n-1-s$ and $w(r)+w(s)=w(r+s)$, 
there is some $c \in \euO_K^\times$ such that $\Phi^{(r)}\Phi^{(s)}
-c\Phi^{(r+s)}\in\pi \euA$. 
\end{proposition}
\begin{proof}
By Proposition \ref{psi-vals} applied successively to $\Psi^{(s)}$ and
$\Psi^{(r)}$, together with Lemma \ref{t-u}(i), we have for any $t
\in \Spn(h)$ that $\Phi^{(r)} \Phi^{(s)} \cdot \lambda_t \in
\pi\euP_L^h$ unless $s \preceq \eua(t)$ and $r \preceq
\eua(H(s,t))=\eua(t)-s$. In particular, if $r \not \preceq p^n-1-s$
then $\Phi^{(r)} \Phi^{(s)} \cdot \lambda_t \in \pi\euP_L^h$ for all
$t\in \Spn(h)$, so that $\Phi^{(r)} \Phi^{(s)} \in \pi \euA$.

Now suppose that $r \preceq p^n-1-s$. Applying Proposition
\ref{psi-vals} to $\Psi:=\Psi^{(r)}\Psi^{(s)} \in \Upsilon^{(r+s)}$, 
we find that the element 
$$ \Phi:=\pi^{-w(r+s)}\Psi=\pi^{w(r)+w(s)-w(r+s)} \Phi^{(r)}\Phi^{(s)} $$
satisfies $v_L(\Phi\cdot \lambda_t)\geq h$ for all $t \in \Spn(h)$, so
that $\Phi \in \euA$. 
Now it follows from (\ref{def-ws}) that $w(r+s) \geq w(r)+w(s)$. Thus
if $w(r)+w(s)\neq w(r+s)$, we have 
$\Phi^{(r)} \Phi^{(s)} \in \pi^{w(r+s)-w(r)-w(s)} \euA \subseteq \pi \euA$. 

It remains to consider the case that $r \preceq p^n-1-s$ and
$w(r)+w(s)=w(r+s)$, so that  
$\Phi=\Phi^{(r)}\Phi^{(s)}$. Since the $\Phi^{(u)}$ form an
$\euO_K$-basis for the order $\euA$,  
we have 
\begin{equation} \label{expand-Phi-prod}
    \Phi^{(r)}\Phi^{(s)} =\sum_{u \in \Spn} c_u \Phi^{(u)}
\end{equation}
for some $c_u \in \euO_K$.
We apply Proposition \ref{psi-vals} on the one hand to
$\Psi=\Psi^{(r)}\Psi^{(s)}$, and on the other hand to each
$\Psi^{(u)}$. This gives the following congruences mod $\pi \euP_L^h$: 
\begin{equation} \label{psi-cong-1}
    \Phi^{(r)} \Phi^{(s)} \cdot \lambda_t \equiv \begin{cases} 
         y_t \lambda_{H(r+s,t)} & \mbox{if } r+s \preceq \eua(t) 
                     \mbox{ and } \epsilon(r+s,t)=0, \\ 
         0                     & \mbox{otherwise,}  \end{cases}
\end{equation}
\begin{equation} \label{psi-cong-2}
   \Phi^{(u)} \cdot \lambda_t \equiv \begin{cases} 
         z_{u,t} \lambda_{H(u,t)} & \mbox{if } u \preceq \eua(t) 
                 \mbox{ and } \epsilon(u,t)=0, \\
         0                     & \mbox{otherwise,}  \end{cases}
\end{equation}
with $y_t$, $z_{u,t}\in \euO_K^\times$. In view of 
(\ref{expand-Phi-prod}), if we multiply (\ref{psi-cong-2}) by $c_u$ and
sum over $u$, we must obtain the same congruence as
(\ref{psi-cong-1}) for each $t$. Thus $c_u \Phi^{(u)} \cdot \lambda_t \in \pi
\euP_L^h$ unless $u=r+s\preceq \eua(t)$ and $\epsilon(r+s,t)=0$, in
which case we have $c_{r+s} z_{r+s,t} \equiv y_t \pmod{\pi \euO_K}$. 
Let $c=c_{r+s}$. Since $c$ is independent of $t$, it follows that
$$ (\Phi^{(r)} \Phi^{(s)} - c \Phi^{(r+s)}) \cdot  \lambda_t \in \pi
\euP_L^h $$
for all $t \in \Spn(h)$. Hence
$\Phi^{(r)}\Phi^{(s)}-c\Phi^{(r+s)}\in\pi \euA$ as required. 
\end{proof}    

\begin{pf-gencount}
Let $\euM$ be the $\euO_K$-submodule of $\euA$ spanned by
$\pi=\pi\Phi^{(0)}$ and the $\Phi^{(s)}$ for $s \in
\Spn\backslash\{0\}$. It is immediate from Proposition \ref{psi-prod} that
$\euM$ is an ideal in $\euA$. Clearly $\euA/\euM \cong \euO_K/\euP_K =
\kappa$, so $\euM$ is a maximal ideal and has residue field
$\kappa$. We claim that $\euM^{n(p-1)+1} \subseteq \pi \euA$,
so that $\euM$ is topologically nilpotent. This will show that every
maximal ideal is contained in $\euM$, so that $\euM$ is in fact the
unique maximal ideal and $\euA$ is a local ring.

To prove the claim, it will suffice to show that if $\Phi^{(s_1)}
\cdots \Phi^{(s_m)} \not \in \pi \euA$ with $s_1, \ldots, s_m \in
\Spn\backslash\{0\}$, then $m \leq n(p-1)$.  For $s=\sum_{i=1}^n
s_{(n-i)}p^{n-i} \in \Spn$, define $|s|=\sum_{i=1}^n s_{(n-i)}$. Thus
for $s \in \Spn\backslash\{0\}$ we have $1 \leq |s| \leq n(p-1)$. By
Proposition \ref{psi-prod}, if $\Phi^{(s_1)} \Phi^{(s_2)} \not \in \pi
\euA$ then $\Phi^{(s_1)} \Phi^{(s_2)} \equiv c \Phi^{(s_1+s_2)} 
\pmod{pi \euA}$ for some $c \in \euO_K^\times$, and $s_1 \preceq
p^n-1-s_2$. Since $s_1$, $s_2 \neq 0$, the latter condition implies that
$s_1+s_2 \in \Spn$ and $0\prec s_1 \prec s_1+s_2$, using Lemma
\ref{preceq}.  Inductively, if
$\Phi^{(s_1)} \cdots \Phi^{(s_m)} \not \in 
\pi \euA$ then 
$$ 0 \prec s_1 \prec s_1 + s_2 \prec \ldots \prec s_1+ \cdots
     +s_m,  $$ 
so that $0 < |s_1| < |s_1 + s_2| < \ldots < |s_1+ \cdots
+s_m|$, which is only possible if $0<m \leq n(p-1)$. This completes the
proof that $\euA$ is a local ring.

Consider now the minimal number of generators of $\euP_L^h$ over
$\euA$. By Nakayama's Lemma, a subset of $\euP_L^h$ is a generating
set if and only if it generates $\euP_L^h / (\euM \cdot \euP_L^h)$
over $\euA/\euM = \kappa$. By Proposition \ref{psi-vals}, $\euM \cdot
\euP_L^h$ is spanned over $\euO$ by $\pi \euP_L^h$ and the elements
$\Phi^{(s)}\cdot \lambda_t$ where $0 \neq s \preceq \eua(t)$ and
$\epsilon(s,t)=0$. Let $u$ correspond to $t$ as in Lemma
\ref{t-u}. Then $\Phi^{(s)} \cdot \lambda_t \equiv y \lambda_{H(u,b)}
\pmod{\pi \euP_L^h}$ with $y \in \euO_K^\times$, and the condition
$\epsilon(s,t)=0$ is equivalent to $D(s,t)=w(s)$, and hence to
$d(u)-d(u-s)=w(s)$. Thus $\euM \cdot \euP_L^h$ is spanned by $\pi
\euP_L^h$ and the $\lambda_{H(u,b)}$ for those $u \in \Spn$ such that
$d(u)=d(u-s)+w(s)$ for some $s$ with $0 \neq s \preceq u$. It follows
that a $\kappa$-basis of $\euP_L^h / (\euM \cdot \euP_L^h)$ is given
by the images of the $\lambda_{H(u,b)}$ for $u \in \DD$, and the
minimal number of generators of $\euP_L^h$ over $\euA$ is $|\DD|$.

Finally, consider the embedding dimension $\dim_\kappa(\euM/\euM^2)$.
Write $A^+$ for the augmentation ideal $\{a\in A:a\cdot 1=0\}$ of
$A$. This is spanned over $K$ by the $\Phi^{(s)}$ for $s \in
\Spn\backslash\{0\}$. Then $\pi\euA\cap A^+= \pi\euM\cap A^+$, since
both are spanned over $\euO_K$ by the $\pi \Phi^{(u)}$ for $u \in
\Spn\backslash\{0\}$. Now $\euM^2$ is spanned over $\euO_K$ by $\pi
\euM$ and the products $\Phi^{(r)} \Phi^{(s)}$ for $r$, $s \in
\Spn\backslash\{0\}$. By Proposition \ref{psi-prod} we have
$\Phi^{(r)} \Phi^{(s)} \in \pi \euA \cap A^+ \subset \pi \euM$ unless
$s\preceq p^n-1-r$ and $w(r)+w(s)=w(r+s)$. Conversely, when $s\preceq
p^n-1-r$ and $w(r)+w(s)=w(r+s)$, we have $\Phi^{(r)} \Phi^{(s)} \equiv
c\Phi^{(r+s)} \pmod{\pi \euM}$ for some $c \in \euO_K^\times$. Now we
may write $u \in \Spn$ as $u=r+s$, where $r$, $s \in
\Spn\backslash\{0\}$ with $s\preceq p^n-1-r$ and $w(r)+w(s)=w(r+s)$,
precisely when $u \not \in \EE$. Thus the images in $\euM/ \euM^2$ of
$\pi$ and the $\Phi^{(u)}$ with $u \in \EE \backslash\{0\}$ form a
$\kappa$-basis of $\euM/\euM^2$. Since $0 \in \EE$, we have
$\dim_\kappa(\euM/\euM^2)=|\EE|$.
\end{pf-gencount}

\section{Applications to Galois Extensions} \label{galois-examples}

In this section, we give some explicit applications of Theorems
\ref{IGMS} and \ref{gencount}, and relate our approach to various
results already in the literature. Except where otherwise stated, we
consider only the classical setting, where $L/K$ is a Galois extension
and $A$ is the group algebra $K[G]$ for $G=\Gal(L/K)$, with its usual
action on $L$. The scaffolds will then be Galois scaffolds in the
sense of Definition \ref{Gal-scaff}. In particular the residue field
$\kappa$ of $K$ will be assumed to be perfect of characteristic $p$,
and the shift parameters will be the (lower) ramification
breaks. Also, the units $u_{i,t}$ in Definition \ref{tol-scaffold}(ii)
will always be $1$.

Our basic examples are the near
one-dimensional extensions constructed in \cite{elder:scaffold}. These
are certain elementary abelian extensions in characteristic $p$. In the
terminology of this paper, the main result of \cite{elder:scaffold} is
that any near one-dimensional extension
admits a Galois scaffold of precision $\infty$. A necessary and 
sufficient condition for the valuation ring $\euO_L$ of a near one-dimensional
extension to be free over $\euA_{L/K}$ is given in \cite[Theorem
  1.1]{byott:scaffold} (see \S\ref{valring} below). 
Theorem \ref{IGMS} of this
paper, applied to near one-dimensional extensions, improves on this result by giving an analogous result for any
fractional ideal of the valuation ring. 

The near one-dimensional extensions include all totally ramified
biquadratic extensions in characteristic $2$, and all totally and
weakly ramified extensions in characteristic $p$. In the next two
subsections, we study these two cases in detail. In a separate paper
\cite{byott-elder2}, we construct a family of elementary abelian
extensions in characteristic $0$ which possess Galois scaffolds and
are the analog of the near one-dimensional extensions. These include
all biquadratic extensions and weakly ramified $p$-extensions
satisfying some mild additional hypotheses. The results of the next
two subsections hold also in characteristic $0$ under these
hypotheses.

\subsection{Biquadratic extensions} \label{biquadratic}

Let $L/K$ be a totally ramified biquadratic extension of local fields
of residue characteristic $2$. When $K$ has characteristic $0$,
the structure of $\euO_L$ over its associated order in $K[G]$ was
studied by Martel \cite{martel}. When $K$ has characteristic $2$ and
has perfect residue field, $\euO_L$ is always free over its associated
order \cite[Corollary 1.4]{byott:scaffold}. These results trivially
extend to fractional ideals $\euP_L^h$ when $h \equiv 0 \pmod{4}$, but
we are not aware of any results for $h \not \equiv 0 \pmod{4}$. In this
subsection, we give analogous
results for arbitrary $h$. We also provide
supplementary information about the number of generators for the
ideals which are not free and the embedding dimensions of the
associated orders. 

\begin{theorem} \label{biquad}
Let $K$ be a local field of characteristic $p=2$ with perfect
residue field.  
Let $L$ be a totally ramified biquadratic extension of $K$ with
lower ramification breaks $b_1$, $b_2$, let $h \in \bZ$, and let
$\euA$ be the associated order of $\euP_L^h$. Then $\euP_L^h$ is free
over $\euA$ if and only if $b_1 \equiv 1 \pmod{4}$, $h\not \equiv 2
\pmod{4}$ or $b_1 \equiv 3 \pmod{4}$, $h\not \equiv 1 \pmod{4}$. In the
cases where $\euP_L^h$ is not free, it requires $3$ generators over
$\euA$. The embedding dimension of $\euA$ is $3$ if $b_1 \equiv 1 
\pmod{4}$, $h \equiv 1 \pmod{2}$ or $b_1 \equiv 3 \pmod{4}$, $h \equiv 0 
\pmod{2}$, and is $4$ otherwise.
\end{theorem}
\begin{proof} 
By \cite[Lemma 5.1]{elder:scaffold}, $L/K$ has a Galois scaffold of
precision $\infty$, so we may apply Theorems \ref{IGMS} and
\ref{gencount}. Recall that $b$ in Theorem \ref{IGMS} satisfies $b
\equiv b_2 \pmod{4}$ and $0 \leq b-h <4$. As the ramification breaks
$b_1$ and $\frac{1}{2}(b_1+b_2)$ of the  
quadratic subextensions $F/K$ of $L/K$ must be odd, we have   
$b_1 \equiv b_2 \equiv 1$ or $3 \pmod{4}$.  Since
the condition $w(s)=d(s)$, together with the sets $\DD$ and $\EE$,
only depends on the residue classes of $h$, $b_1$ and $b_2 \bmod 4$, there is no loss of generality in assuming that
$b_1=b_2=b=1$ or $3$ and $b-3 \leq h \leq b$. Then $\eub(s)=bs$ and
the values of $d(s)$ and $w(s)$ are as shown in Table 1, which also
shows the sets $\DD$ and $\EE$ occurring in Theorem \ref{gencount}.
To obtain the $w(s)$, note that $w(0)=d(0)=0$,
$w(1)=\min(d(1)-d(0),d(3)-d(2))$, $w(2)=\min(d(2)-d(0),d(3)-d(1))$,
$w(3)=d(3)-d(0)$.

From Table 1, we have $w(s)=d(s)$ for all $s$ except in the
cases $b=1$, $h=-2$ and $b=3$, $h=1$. The criterion for
$\euP_2^h$ to be free then follows from Theorem \ref{IGMS}. In the cases where 
$\euP_2^h$ is not free, $|\DD|=3$, so that $\euP_2^h$ requires 3
generators over $\euA$ by Theorem \ref{gencount}. The cardinalities of
the sets $\EE$ in Table 1 show that the embedding dimension is as stated.
\end{proof}

\begin{table}  
\centerline{ 
\begin{tabular}{|r|r|r r r r|  r r r r | c | c |} \hline
  &   &  \multicolumn{4}{c|}{$d(s)$}  &  \multicolumn{4}{c|}{$w(s)$}  
     &   &  \\  \cline{3-10}
  & & \multicolumn{4}{c|}{$s=$}  &  \multicolumn{4}{c|}{$s=$} & 
  &  \\
 $b$ & $h$ & 0 & 1 & 2 & 3 & 0 & 1 & 2 & 3 & $\DD$ & $\EE$ \\ \hline   
1 & 1 & 0 & 0 & 0 & 0 & 0 & 0 & 0 & 0 & $\{0\}$ & $\{0,1,2\}$ \\ 
1 & 0 & 0 & 0 & 0 & 1 & 0 & 0 & 0 & 1 & $\{0\}$ & $\{0,1,2,3\}$  \\ 
1 & -1 & 0 & 0 & 1 & 1 & 0 & 0 & 1 & 1 & $\{0\}$ & $\{0,1,2\}$  \\ 
1 & -2 & 0 & 1 & 1 & 1 & 0 & 0 & 0 & 1 & $\{0,1,2\}$ & $\{0,1,2,3\}$  \\ \hline
3 & 3 & 0 & 0 & 1 & 2 & 0 & 0 & 1 & 2 & $\{0\}$ & $\{0,1,2,3\}$  \\ 
3 & 2 & 0 & 1 & 1 & 2 & 0 & 1 & 1 & 2 & $\{0\}$ & $\{0,1,2\}$ \\ 
3 & 1 & 0 & 1 & 2 & 2 & 0 & 0 & 1 & 2 & $\{0,1,2\}$ & $\{0,1,2,3\}$  \\ 
3 & 0 & 0 & 1 & 2 & 3 & 0 & 1 & 2 & 3 & $\{0\}$ & $\{0,1,2\}$   \\  \hline
\end{tabular}
}  
\vskip5mm

\caption{The biquadratic case: $d(s)$, $w(s)$, $\DD$ and $\EE$.} 
\end{table}

\subsection{Weakly ramified $p$-extensions} \label{weakly}

A Galois extension $L/K$ of local fields with Galois group $G$ is said
to be weakly ramified if its second ramification group $G_2$ is
trivial. Then $\euP_L$ is free over the group ring $\euO_K[G]$, and
$\euO_L$ is free over the order $\euO_K[G][\pi^{-1} \sum_{g \in G_0}
  g]$, where $\pi$ is a uniformizing parameter of $K$ and $G_0$ is the
inertia subgroup of $G$ (see for instance \cite{johnston}). Moreover,
a fractional ideal $\euP_L^h$ is free over $\euO_K[G]$ if and only if
$h \equiv 1 \pmod{|G_1|}$ \cite[Theorem 1.1]{koeck}. Thus if $L/K$ is
totally and weakly ramified of degree $p^n$ then the ideals $\euP_L^h$ are
free over their associated orders when $h \equiv 0$ or $1
\pmod{p^n}$. For other values of $h$, nothing seems to be known when
$n>1$ beyond the fact that $\euP_L^h$ cannot be free over $\euO_K[G]$.
The case $n=1$ is covered by Ferton's result \cite{ferton:ideals}
mentioned in Remark \ref{deg-p-consequences}.  

In this subsection, we will give detailed information on $\euP_L^h$
for all $h$ (and arbitrary $n$), assuming that $K$ has characteristic
$p$ and has perfect residue field. We will determine precisely when
$\euP_L^h$ is free over its associated order $\euA$, and will obtain
supplementary information about the minimal number of generators of
$\euP_L^h$ over $\euA$ and the embedding dimension of
$\euA$. Our results will be expressed in terms of combinatorial
properties of the base-$p$ digits of numbers closely related to $h$.

Let $K$ be as just described, and let $L/K$ be a totally and weakly
ramified extension of degree $p^n$. Thus $L/K$ has ramification breaks
$b_1= \cdots = b_n=1$, and its Galois group must be elementary
abelian. Moreover, $L/K$ admits a Galois scaffold of precision
$\infty$. For $n=1$, we have already seen this in Example \ref{deg-p},
and for $n \geq 2$ it follows from \cite[Lemma
  5.3]{elder:scaffold}. We can therefore apply Theorems \ref{IGMS} and
\ref{gencount}.

We first define some notation. For $s= \sum_{i=1}^n s_{(n-i)}p^{n-i} =
\sum_{j=0}^{n-1} s_{(j)} p ^j \in \Spn$, set 
$$ \alpha(s) = | \{ j : 1 \leq j \leq n-1, 
         j>v_p(s), s_{(j)} \neq p-1 \}|, $$
$$  \beta(s) = \max\{c :  0 \leq  c < n-v_p(s), \,  s_{(n-1)}= \ldots =
                               s_{(n-c)}={\textstyle{\frac{1}{2}}}(p-1)\}, $$
where the maximum is to be interpreted as $0$ if no such $c$ exists,
and 
$$ \gamma(s) = \begin{cases} 1 & \mbox{if } p=2 \mbox{ and } s=2^{n-1}, \cr
                             0 &  \mbox{otherwise.} \end{cases} $$
Thus $\alpha(s)$ is the number of base-$p$ digits of $s$ which are not
equal to $p-1$, including any leading $0$'s
$s_{(n-1)}=\cdots=s_{(m)}=0$, but excluding the last
nonzero digit $s_{(v)} \neq 0$ for $v=v_p(s)$ and any trailing $0$'s
$s_{(v-1)}=\cdots=s_{(0)}=0$.
Also, $\beta(s)$ is the number of leading base-$p$ digits (including leading
$0$'s but excluding the last nonzero digit) which are equal to
$\frac{1}{2}(p-1)$. In particular, $\beta(s)=0$ if 
$s<\frac{1}{2} p^{n-1}(p-1)$ or if $p=2$.

For $0 \leq j \leq n-1$, we define 
$$   \left\lfloor s \right\rfloor_j =   p^j \left\lfloor
      \frac{s}{p^j} \right \rfloor = \sum_{i=j}^{n-1} s_{(i)}p^i, $$
and
$$ \left \lceil s \right \rceil_j =  
  p^j \left \lceil \frac{s}{p^j} \right \rceil 
  = \begin{cases} \left \lfloor s \right \rfloor_j = s & \mbox{if }
                                   s \equiv 0 \pmod{p^j} \cr
           \left \lfloor s \right \rfloor_j  + p^j &
           \mbox{otherwise. } \end{cases}  $$ 
\begin{theorem} \label{weak}
Let $L/K$ be a totally and weakly ramified extension of degree $p^n$
in characteristic $p$. Let $h \in \mathbb{Z}$.

\begin{itemize}
\item[(i)] If $h \equiv 1 \pmod{p^n}$ then $\euP_L^h$ is free over its
  associated order, and this order has embedding dimension $n+1$.

\item[(ii)] If $h \not \equiv 1 \pmod{p^n}$, let $h' \equiv h \pmod{p^n}$
  with $2 \leq h' \leq p^n$, and write $m=h'-1$ and 
$k=\max(m,p^n-m)$. Then
\begin{itemize}
\item[(a)] $\euP_L^h$ is free over its associated order $\euA$ if and
  only if $h' \geq 1 + \frac{1}{2}p^n$;
\item[(b)] if $ \euP_L^n$ is not free, the minimal number of
  generators of $\euP_L^h$ as a module over $\euA$ is $2+\alpha(m)-\beta(m)$;
\item[(c)] the embedding dimension of $\euA$ is
  $n+2+\alpha(k)-\gamma(k)$.
\end{itemize}
\end{itemize}     
(Note that when $h=0$ we have $h'=p^n$ so that, in particular, $\euO_L$
is free over its associated order; cf.~Remark \ref{weakly-val} below.) 
\end{theorem}
\begin{proof}
As $b_i=1$ for each $i$, we have $\eub(s)=s$. Without loss of
generality, we suppose that $2\leq h \leq p^n+1$. Thus $b=p^n+1$, and
$h'=h$ in (ii). We then have  
\begin{equation} \label{d-wr}
   d(s)  = \left\lfloor \frac{1+p^n+s-h}{p^n} \right \rfloor  
      = \begin{cases}
               1 & \mbox{if } s \geq m; \cr
               0 & \mbox{if } s < m. 
        \end{cases} 
\end{equation}

\noindent 
(i) If $h=p^n+1$ we have $d(s)=0$ for all $s$, and hence $w(s)=0$ for all $s$ as
well. Thus $\euP_L^{h}$ is free over its associated order $\euA$ by
Theorem \ref{IGMS}. In Theorem \ref{gencount}, we have  
$\DD=\{0\}$, $\EE=\{0,1,p,\ldots, p^{n-1}\}$, so that $\euA$ has
embedding dimension $n+1$.
\smallskip

\noindent 
(ii) Now let $2 \leq h \leq p^n$. We first determine the  $w(s)$;
for any $s \in \mathbb{S}_{p^n}$ we have 
\begin{eqnarray*}
 w(s)=1 & \LRA & d(u)=1 \mbox{ and } d(u-s)=0 
                           \mbox{ for all } u \succeq s \\
        & \LRA & u \geq m \mbox{ and } u-s<m 
                           \mbox{ for all } u \succeq s \\
        & \LRA & s \geq m \mbox{ and } (p^n-1)-s<m \\
        & \LRA & s \geq \max(m,p^n-m),  
\end{eqnarray*}
so that 
\begin{equation} \label{w-wr}
      w(s)  = \begin{cases}
               1 & \mbox{if } s \geq k; \cr
               0 & \mbox{if } s < k. 
        \end{cases} 
\end{equation}
Note that $\frac{1}{2}p^n \leq k \leq  p^n-1$. 
\smallskip

\noindent
(a) From (\ref{d-wr}) and (\ref{w-wr}) we have 
$$ d(s)=w(s) \mbox{ for all } s \in \Spn \LRA k=m \LRA h \geq 
     1 + \textstyle{\frac{1}{2}}p^n. $$

\noindent 
(b) Let $2 \leq h < 1+\frac{1}{2}p^n$. Then $0<m<k$. It is immediate from  
(\ref{d-wr}) and (\ref{w-wr}) that $\DD$ contains $0$ and
$m$. Moreover, if $0<u<m$ or $u \geq k$ then $u \not \in \DD$ since
$d(u)=d(0)+w(u)$.  

We need to show that there are $\alpha(m)-\beta(m)$ elements $u \in
\DD$ with $m<u<k$. Let $j=v_p(u)$. Then $0 \leq j \leq n-1$ and
$u_{(j)} \neq 0$. If $u-p^j \geq m$ then, taking $s=p^j$, we have
$d(u-s)=1$ and $0 \prec s \preceq u$, so that $u \not \in
\DD$. Conversely, if $u \not \in \DD$, then there is some $s$ with
$d(u-s)=1$ and $0 \prec s \preceq u$. Since $v_p(u)=j$, we must have
$s \geq p^j$ and hence $m \leq u-s \leq u-p^j$. It follows that $u \in
\DD$ if and only if $u-p^j<m$. But as $j=v_p(u)$, we have $u-p^j<m<u$
if and only if $u=\lceil m \rceil_j$ and $v_p(m)<j$. We conclude that,
for each $j>v_p(m)$, there is at most one $u \in \DD$ with $v_p(u)=j$
and $u>m$, namely $u=\lceil m \rceil_j$; such a $u$ exists if and only
if $\lceil m \rceil_j<k$ and $v_p\left(\lceil m
\rceil_j\right)=j$. Since $j>v_p(m)$, the latter condition is
equivalent to $m_{(j)} \neq p-1$, and the number of $j$ for which this
occurs is $\alpha(m)$. We claim that, amongst these, there are
$\beta(m)$ values of $j$ for which $\lceil m \rceil_j \geq k$.

We count the $j\leq n-1$ such that 
\begin{equation} \label{beta-count}
 m_{(j)} \neq p-1 \mbox{ and } \lceil m \rceil_j \geq k. 
\end{equation}
Any such $j$ automatically satisfies $j>v_p(m)$ since if $j \leq
v_p(m)$ then $\lceil m \rceil_j=m<k$.  We distinguish two
cases. Firstly, we consider the special case where the base-$p$ digits
of $m$ are all $\frac{1}{2}(p-1)$, possibly followed by a block of $0$'s. Thus
$m=\frac{1}{2}(p-1)(p^{n-1}+ \cdots + p^v)$ where $v=v_p(m) \geq 0$.
In this case, $\beta(m)=n-v-1$ and $k=m+p^v$. If $j \geq
n-\beta(m)=v+1$ then $\lceil m \rceil_j = \lceil k \rceil_j >k$, and
of course $m_{(j)} = \frac{1}{2}(p-1) \neq p-1$, while if $j \leq v$
then $\lceil m \rceil_j = m <k$. Thus there are $\beta(m)$ values of
$j$ satisfying (\ref{beta-count}) in this case. Secondly, suppose we
are not in this special case, and let $c= \beta(m)$. Then $0 \leq c
\leq n-1$ and $v_p(m)<n-c$. Moreover, since $m<\frac{1}{2}p^n$
(because $m<k$) and we are not in the first case, we have 
$m_{(n-c-1)}< \frac{1}{2}(p-1)$.  If $p \neq 2$ then 
$m_{(n-c-1)} \leq \frac{1}{2}(p-3)$ and we may write
$m=\frac{1}{2}(p-1)(p^{n-1}+ \cdots + p^{n-c})+r$ with 
$0<r<\frac{1}{2}(p-3)p^{n-c-1} + p^{n-c-1}
=\frac{1}{2}(p-1)p^{n-c-1}$. Then $\lceil m \rceil_j = \lceil k 
\rceil_j >k$ if $j\geq n-c$, and $\lceil m \rceil_j \leq \lfloor k
\rfloor_j \leq k$ if $j<n-c$. Thus there are again $\beta(m)$
values of $j$ satisfying (\ref{beta-count}). Finally, if $p=2$ then
$\beta(m)=0$ and, since $m<2^{n-1}<k$, we have $\lceil m
\rceil_j < k$ for all $j<n$, yet again giving the required conclusion.
\smallskip

\noindent
(c) By Theorem \ref{gencount}, the embedding dimension of $\euA$ is
     $|\EE|$ where 
$$ \EE = \{ u \in \Spn \, : \, w(u)>w(u-s)+w(s) \mbox{ for all } s \in
\Spn \mbox{ with } 0 \prec s \prec u \}. $$ This set will be unchanged
on replacing $h$ by $p^n+2-h$, since both give the same value for $k$
and hence the same sequence $w(s)$. Certainly $\EE$ contains the $n+1$
elements $0$, $1$, $p, \ldots, p^{n-1}$, and no other elements $u<k$. It also contains $k$
since $w(k)=1$ and $w(s)=0$ for $s<k$. Note that $k > p^{n-1}$ except
in the case $p=2$, $k=2^{n-1}$ (corresponding to $h=2^{n-1} +
1$). Thus the number of elements $u \in \EE$ with $u \leq k$ is $n+2
-\gamma(k)$. The proof will be complete if we show that there are
precisely $\alpha(k)$ elements $u \in \EE$ with $u>k$. But if $u>k$
and $v_p(u)=j$ then, arguing as in (b) above, $u \in \EE$ if and only
if $u-p^j<k$, and the number $u$ satisfying this condition is
$\alpha(k)$.
\end{proof}
\begin{remark}  \label{gp-ring}
When $h \equiv 1 \pmod{p^n}$, the associated order $\euA$ is just the
group ring $\euO_K[G]$, and 
its maximal ideal is $\euM = \euP_K + I$ where $I$ is the augmentation
ideal of $\euO_K[G]$. Thus $\euM/\euM^2$ is generated as a
$\kappa$-vector space by the $n+1$ elements $\pi$, $\sigma_1-1,
\ldots, \sigma_n-1$, where $\pi \in K$ with $v_K(\pi)=1$ and
$\sigma_1, \ldots, \sigma_n$ is any set of generators of $G$.
\end{remark}
\begin{remark} \label{weakly-val}
In the case $h =0$, we have $h'=p^n$, so that
$k=m=p^n-1$ and $\alpha(k)=\gamma(k)=0$ (unless $p^n=2$, when
$\gamma(k)=1$).  Hence $\euO_L$ is free over its associated order
$\euA_{L/K}$, as we already know from \cite[Lemma 5.3]{elder:scaffold}
and \cite[Theorem 1.1]{byott:scaffold}. Moreover, $\euA_{L/K}$ has
embedding dimension $n+2$ (or $n+1$ when $p^n=2$).  One can check
directly that
$$    d(s)  = w(s) = \begin{cases}
               1 & \mbox{if } s=p^n-1;  \cr
               0 & \mbox{if } 0\leq s < p^n-1, 
        \end{cases} $$
so that $\EE=\{0,1,p,p^2, \ldots ,p^{n-1}, p^n-1\}$. In fact, 
$$ \euA = \euO_K[G][\pi^{-1} \Sigma], $$
where $v_K(\pi)=1$ and $\Sigma = \sum_{\sigma \in G} \sigma$ is the
trace element of $K[G]$.  Thus, with the notation of Remark
\ref{gp-ring}, $\euM/\euM^2$ is generated by $\pi$, $\sigma_1-1,
\ldots, \sigma_n-1, \pi^{-1}\Sigma$. 
\end{remark}

To give some idea of the range of complexity occurring in the Galois
module structure of ideals for wildly ramified extensions, we record
the maxima and minima of the number of generators and the
embedding dimension. We are not aware of any similar results in the
Galois module literature beyond the discussion of the degree $p$ case
in \cite{desmit:2}.

\begin{corollary} 
Let $L/K$ be as in Theorem \ref{weak}, and let $\euA$ be the associated order of
$\euP_L^h$. 
\begin{itemize}
\item[(i)]
\begin{itemize}
\item[(a)] When $p>2$, the maximal number of generators required for
  $\euP_L^h$ over $\euA$ is $n+1$. The minimal number of
  generators in cases where $\euP_L^h$ is not free is $2$. This occurs,
  for example, when $h=\frac{1}{2}(p^n+1)$. 
\item[(b)] When $p=2$ and $n>1$, the maximal number of generators for
  $\euP_L^h$ over  $\euA$ is again $n+1$. There are no $\euP_L^h$
requiring precisely $2$ generators. Precisely $3$ generators are required,
for example, if $h=2^{n-1}$. 
\end{itemize}
\item[(ii)]
\begin{itemize}
\item[(a)] When $p> 2$, the embedding dimension of $\euA$ can take any
  value between $n+1$ and
$2n+1$. The minimum $n+1$ occurs only for $h \equiv 1 \pmod{p^n}$. The
  value $n+2$ occurs, for example, if $h=2$ or $h=p^n$. The maximal
  value $2n+1$ occurs, for example, if $h=\frac{1}{2}(p^n+1)$.
\item[(b)] When $p=2$, the minimum embedding dimension $n+1$ is attained 
only for $h \equiv 1$ and $2^{n-1} + 1 \pmod{2^n}$. The maximum is $2n$,
attained only for $h \equiv 2^{n-1}$ and $2^{n-1}+2 \pmod{2^n}$.
\end{itemize}
\end{itemize}
\end{corollary}

\begin{remark}
If $L/K$ is any extension of local fields (not necessarily Galois)
with an action of an algebra $A$ admitting a scaffold of $\tol \geq
2p^n-1$ whose shift parameters satisfy $b_i \equiv 1 \pmod{p^i}$, then
we can reduce to the case $b_i=1$ for all $i$ by Remark
\ref{adjust-shift}. We will then obtain the same sequences $d(s)$ and
$w(s)$ as in the proof of Theorem \ref{weakly}, so the conclusions of
Theorem \ref{weakly} will still hold.  In particular, this gives an
alternative approach to Theorem \ref{biquad} in the case that $b_1
\equiv 1 \pmod{4}$.
\end{remark}

\begin{remark}
There is an arithmetic interpretation of the fact that the sequence
$w(s)$ is unchanged on replacing $h$ by $p^n+2-h$. Let $L/K$ be a
totally and weakly ramified extension of degree $p^{n}$. Then its
inverse different is $\euP_L^{2(1-p^n)}$. For any $m \in \mathbb{Z}$,
the ideals $\euP_L^{1-p^n+m}$ and $\euP_L^{1-p^n-m}$ are therefore
mutually dual under the trace pairing. Thus, for any $h \in
\mathbb{Z}$, the ideals $\euP_L^h$ and $\euP_L^{2-2p^n-h} \cong
\euP_L^{p^n+2-h}$ are mutually dual. When $h \equiv 1 \pmod{p}$, the
ideal $\euP_L^h$ is isomorphic to its dual, and is free over the group
ring $\euO_K[G]$. If $p=2$ and $h \equiv 1+2^{n-1}$ mod $2^n$, the
ideal $\euP_L^h$ is again isomorphic to its dual, and is free over its
associated order $\euA$; in this case $\euA \neq \euO_K[G]$, although
$\euA$ attains the minimal embedding dimension $n+1$. In the remaining
case $2h \not \equiv 2$ mod $p^n$, the mutually dual ideals $\euP_L^h$
and $\euP_L^{2-2p^n-h}$ are not isomorphic; they have the same
associated order, since both give rise to the same sequence $w(s)$,
but one ideal is free over this order while the other is not.
\end{remark}

\subsection{More on the valuation ring}
 \label{valring}

In this subsection, we discuss how Theorem \ref{IGMS} is related to a
result of Miyata \cite{miyata:cyclic2} and to our previous work in
\cite{byott:QJM} and \cite{byott:scaffold}. This will lead to a
strengthening of \cite[Corollary 1.2]{byott:scaffold}.

We first recall Miyata's result. Let $K$ be a finite extension of
$\mathbb{Q}_p$ containing a primitive $p^n$th root of unity, and let
$L=K(\sqrt[p^n]{a})$ be an extension of degree $p^n$, where $a \in K$
and $p \nmid v_K(a-1)$. Recall that $r(x)$ denotes the least
non-negative residue mod $p^n$ of an integer $x$. We set
$t_0=r(v_K(a-1))$. Miyata \cite[Theorem 5]{miyata:cyclic2} shows that
$\euO_L$ is free over its associated order $\euA_{L/K}$ if and only if
the following condition holds: $t_0+r(it_0)-r(ht_0)>0$ for all
integers $h$, $i$, $j$ such that $0 \leq h \leq i \leq j <p^n$,
$i+j=p^n-1+h$ and $p \nmid \binom{i}{h}$. This can be interpreted as a
condition on the ramification breaks $b_1$, \ldots, $b_n$ of $L/K$,
since, writing $b=r(-v_K(a-1))=p^n-t_0$, we have $b_i \equiv b
\pmod{p^i}$ for $1 \leq i \leq n$.

Miyata's condition was reformulated in \cite{byott:QJM}, where it was
used to deduce a more transparent (but less complete) criterion:
$\euO_L$ is free over $\euA_{L/K}$ if $b$ divides $p^m-1$ for some $m
\in \{1, \ldots, n\}$. The converse is not always true when $n \geq
3$, but for $n=2$ the converse does hold. Thus, for $n=2$, we have
that $\euO_L$ is free if and only if $b$ divides $p^2-1$. This is
closely analogous to the result \cite{ferton} for $n=1$ (cf.~Example 
\ref{deg-p-consequences}): $\euO_L$ is
free if and only if $b \mid (p-1)$.

In \cite{byott:scaffold}, we considered near one-dimensional
extensions $E/F$, and gave a criterion \cite[Theorem
  2.3]{byott:scaffold} for $\euO_E$ to be free over its associated
order. In the notation of the present paper, this criterion is just
the condition that $w(s)=d(s)$ for all $s$, and the result is a
special case of Theorem \ref{IGMS} (with $h=0$). We also showed
\cite[Lemma 2.4]{byott:scaffold} that this criterion was equivalent to
Miyata's, as reformulated in \cite{byott:QJM}. (This is despite the
fact that Miyata's extensions are cyclic in characteristic $0$ and the
near one-dimensional extensions are elementary abelian in
characteristic $p$).

Now, given that $b_i \equiv b \pmod{p^i}$ for $1 \leq i
\leq n$, the equivalence of Miyata's condition and our condition
$w(s)=d(s)$ is a purely numerical statement, depending only on the
parameter $b$. We may therefore combine it with Theorem \ref{IGMS} (in the
case $h=0$) whenever we have an extension $L/K$ admitting a scaffold
with high enough precision and suitable shift parameters. 
We therefore obtain the following result:

\begin{theorem} \label{val-ring-crit}
Let $L/K$ be a totally ramified extension of local fields of degree
$p^n$. Let there be an $A$-scaffold on $L$ with 
shift parameters $b_1,\ldots ,b_n$ that satisfy $b_i\equiv b_n
\pmod{p^i}$ for all $i$ and with precision   
$\tol \geq r(b_n)$. Then $\euO_L$ is free over
  its associated order in $A$ if $r(b_n) \mid (p^m-1)$ for some $m
  \in \{1, 2, \ldots, n\}$. Conversely, if $n\leq 2$ and $\tol \geq 
p^n+r(b_n)$ then $\euO_L$ is free only if $r(b_n) \mid (p^n-1)$.
\end{theorem}

\begin{remark}
We reiterate that, in the case of near one dimensional extensions,
Theorem \ref{val-ring-crit} is \cite[Corollary
  1.2]{byott:scaffold}. The new feature here is that the same
statement holds for any extension (not necessarily Galois, and not
necessarily in characteristic $p$), provided that it admits a scaffold
of high enough precision whose shift parameters satisfy the stated
congruences. These congruences automatically hold for Galois scaffolds
on abelian extensions, cf.~Remark \ref{galois-case}, but also for the
inseparable examples in \S\ref{inseparable} below, where there
is a single shift parameter $b$. 
\end{remark} 

One question which remains unanswered is whether (or under what
conditions) Miyata's cyclic extensions admit a Galois scaffold of
sufficiently high precision for Theorem \ref{IGMS}(ii) to be
applicable. If this were the case, then Miyata's result could be viewed as
particular instance of ours. We hope to return to this question in
future work.

\subsection{A result on the inverse different}

Let $L/K$ be a totally ramified Galois extension of degree $p^n$, with
abelian Galois group. In \cite[Theorem 3.10]{byott:JNTB}, it was shown
that, under a rather mild technical hypothesis, the inverse different
$\euD^{-1}_{L/K}$ of $L/K$ cannot be free over its associated order
unless the ramification breaks satisfy the congruence $b_i \equiv -1
\pmod{p^n}$. (Note that the modulus here is $p^n$, and not $p^i$.) Only
the characteristic $0$ case was considered in \cite{byott:JNTB}, but
the same argument works in characteristic $p$. We will now show that,
if $L/K$ admits a suitable Galois scaffold, this necessary condition
for freeness is also sufficient, and we note an interesting consequence
for the order $\euA$.

\begin{theorem}  \label{inv-diff-thm}
Let $L/K$ be an abelian extension of degree $p^n$ which admits a
Galois scaffold of precision $\tol \geq 2 p^n -1$. Then
$\euD^{-1}_{L/K}$ is free over its associated order $\euA$ if and only
the ramification breaks satisfy $b_i \equiv -1 \pmod{p^n}$ for $1 \leq
i \leq n$. If this occurs, then $\euA$ is also the associated order of
the valuation ring $\euO_L$, and $\euA$ is a Hopf order in the Hopf
algebra $K[G]$, where $G=\Gal(L/K)$.
\end{theorem}
\begin{proof}
The condition (3.11) in \cite[Theorem 3.10]{byott:JNTB} is only
required to ensure that $\euA$ is a local ring (or, equivalently, that
$\euD_{L/K}^{-1}$ is indecomposable as an
$\euO_K[G]$-module). However, this is guaranteed by Theorem
\ref{gencount} under our hypothesis on $\tol$. It follows that
$\euD_{L/K}^{-1}$ cannot be free over $\euA$ unless $b_i \equiv -1
\pmod{p^n}$ for all $i$.

Conversely, suppose that $b_i \equiv -1 \pmod{p^n}$ for all $i$. Then
Hilbert's formula for the different \cite[IV\S2 Prop 4]{serre:local}
gives $\euD^{-1}_{L/K}= \euP_L^{-w}$ with $w \equiv 0 \pmod{p^n}$, so
that $\euD_{L/K}^{-1}=\delta \euO_L$ for some $\delta \in K$.  It
follows that $\euD_{L/K}^{-1}$ is isomorphic to $\euO_L$ as an
$\euO_K[G]$-module. Thus both $\euO_L$ and $\euD_{L/K}^{-1}$ have the
same associated order $\euA$, and if either of them is free over
$\euA$ then so is the other. Now since the assumption on the $b_i$ implies the
weaker congruences $b_i \equiv -1 \pmod{p^i}$, it follows from Example
\ref{bi-1} that $\euO_L$, and hence also $\euD_{L/K}^{-1}$, is indeed
free over $\euA$.  Finally, as $\euD_{L/K}^{-1} = \delta\euO_L$
with $\delta \in K$, and $\euA$ is a local ring, $\euA$ must be a Hopf
order in $K[G]$ by work of Bondarko \cite[Theorem~A and
  Prop.~3.4.1]{bondarko}. 
\end{proof}

\begin{corollary} \label{inv-diff-cong}
Let $L/K$ be an abelian extension of degree $p^n$ which admits a
Galois scaffold of precision $\tol \geq 2 p^n -1$. If the largest
ramification break $b_n$ satisfies the congruence $b_n \equiv -1
\pmod{p^n}$, then we have $b_i \equiv -1 \pmod{p^n}$ for all $i$.
\end{corollary}
\begin{proof}
By the Hasse-Arf Theorem (see Remark \ref{galois-case}), the
hypothesis $b_n \equiv -1 \pmod{p^n}$ ensures that $b_i \equiv -1
\pmod{p^i}$ for all $i$ (which is weaker than the desired conclusion).
But then, on the one hand, it follows that $\euO_L$ is free over its
associated order $\euA$ by Example \ref{bi-1}. On the other hand, from
Hilbert's formula for the different, we again have
$\euD_{L/K}^{-1}=\delta \euO_L$ for some $\delta \in K$. Thus
$\euD_{L/K}^{-1}$ also has associated order $\euA$, and is free over
$\euA$. Hence, by Theorem \ref{inv-diff-thm}, we have the stronger
congruence $b_i \equiv -1 \pmod{p^n}$ for all $i$.
\end{proof}

One can easily construct elementary abelian extensions whose
ramification breaks satisfy $b_i \equiv -1 \pmod{p^i}$ for all $i$,
but do not satisfy $b_i \equiv -1 \pmod{p^n}$ for all $i$. Corollary
\ref{inv-diff-cong} therefore shows that certain realizable sequences
of ramification breaks preclude the existence of a Galois scaffold of
high precision.

\section{Purely inseparable extensions} \label{inseparable}

The purpose of this section is to provide an example of a particularly
natural scaffold (with precision $\tol=\infty$) in the setting of purely inseparable extensions, and
since the results of \S\ref{IGMS-sect} are therefore applicable,
submit the topic of generalized Galois module structure in purely
inseparable extensions for further study.

The divided power Hopf algebra $\A(n)$ of dimension $p^n$ (see
Definition \ref{div-pow} below) is standard example of a Hopf algebra
over a field $K$ of characteristic $p>0$. We will prove the following
result:

\begin{theorem} \label{insep} 
Let $K$ be a local field of characteristic $p>0$,
and let $L$ be any totally ramified and purely inseparable extension of $K$
of degree $p^n$. Let $b$ satisfy $0<b<p^n$ and $\gcd(b,p)=1$.
Then there is an action of $\A(n)$ on $L$ which makes $L$ into an
$\A(n)$-Hopf Galois extension of $K$, and which admits an
$\A(n)$-scaffold with unique shift parameter $b$ and with precision
$\tol=\infty$. 
\end{theorem}

This means that we can study generalized Galois module structure
questions for each of these actions of $\A(n)$: the valuation ring
$\euO_L$ of $L$, or more generally any fractional ideal $\euP_L^h$, is
a module over its associated order in $\A(n)$ under each action,
and, as before, we can ask if it is free, how many generators are
required if it is not, and what the embedding dimension of the
associated order is. The answers to these questions are given in terms
of $b$ by Theorems \ref{IGMS} and \ref{gencount}, and so will be
identical to those for any Galois extension of degree $p^n$ admitting
a Galois scaffold of high enough precision and having lower
ramification breaks $b_i \equiv b \pmod{p^i}$ for $1 \leq i \leq n$.
In particular, it follows from Theorem \ref{val-ring-crit} that
$\euO_L$ will be free over its associated order if $b$ divides $p^m-1$
for any $m \leq n$ (and conversely for $n=1$, $2$).

The material in this section is partly based on discussions with Alan Koch.

\subsection{Hopf Galois structures}

Let $L/K$ be a finite extension of fields, and let $H$ be a
cocommutative $K$-Hopf algebra with comultiplication $\Delta: H \to H
\otimes H$, augmentation (or counit) $\epsilon: H \to K$ and antipode
$\sigma: H \to H$.  We say that $L$ is an $H$-module algebra if there
is a $K$-linear action of $H$ on $L$ such that the following hold: for
all $h \in H$ and $s$, $t \in L$,
\[  \mu (\Delta(h)(s \otimes t) )= h(st)  \]
where  $\mu$ is the multiplication map $L \otimes L \to L$; and 
\[   h \cdot 1 = \epsilon(h) 1 \text{  for all } h \text{  in  } H.\]
Then $L/K$ is an $H$-Hopf Galois extension if $L$ is an $H$-module
algebra and the map 
$$ L\otimes_K H \to \End_K(L), $$
given by $(s \otimes h)(t) \mapsto sh( t)$ for $h \in H$ and $s$, $t \in
L$, is a bijection.   

This notion, defined (in dual form) in \cite{CS69}, extends the
classical concept of a finite Galois extension of fields:  if $L/K$ is
Galois with group $G$, then the map 
\[ L \otimes_K K[G] \to \End_K(L) \] 
is bijective.  

An early example of a class of Hopf Galois extensions was furnished by
finite primitive purely inseparable field extensions.  Let $K$ be a
field of characteristic $p > 0$ and let $L = K(x)$ with $x^{p^n} = a$,
where $a \in K$ but $a^{1/p} \not \in K$.  Note that $x^{p^n} -a$ is
irreducible. Then $L$ is called a primitive extension of $K$ of
exponent $n$.  Associated with a primitive extension $L/K$ of exponent
$n$ are higher derivations, or unital Hasse-Schmidt derivations, of
length $p^n$.  A higher derivation on $L/K$ is a sequence
\[ \mathcal{D} = (D_0 = 1, D_1, \ldots, D_{p^n-1})\]
of $K$-homomorphisms from $L$ to $L$ such that for all $m$ and for all
$a, b$ in $L$,  
\[ D_m(ab) = \sum_{i=0}^m  D_i(a)D_{m-i}(b) \]
and $D_i(a) = \delta_{i, 0}a$ for all $a \in K$.  (Unital means $D_0
= 1$:  see \cite{He13}.) In particular, $D_1$ is a derivation of $L$.
The set of all $a \in L$ so that $D_i(a)= 0$ for all $i > 0$ is the
field of constants of $\mathcal{D}$ (which contains $K$).  

The significance of higher derivations in inseparable field theory
stems in part from characterizations of finite modular purely
inseparable field extensions $L/K$.  A  finite purely inseparable
field extension $L$ of $K$ is modular if $L$ is isomorphic to a tensor
product $K(x_1) \otimes  \ldots \otimes K(x_r)$ of primitive
extensions.  Sweedler \cite{Sw68} characterized a finite modular
extension as one for which $K$ is the field of constants of all higher
derivations on $L/K$. 

Higher derivations of purely inseparable field extensions can arise from actions of divided power Hopf algebras, defined as follows.

\begin{definition} \label{div-pow} (c.f. \cite[5.6.8]{Mo93})  Let $K$
  be a field of 
  characteristic $p$.  The \emph{divided power} $K$-Hopf algebra of
  dimension $p^n$ is the $K$-vector space $\mathcal{A}(n)$ of
  dimension $p^n$  with basis  $t_0, t_1,  \ldots, t_{p^n-1}$.
  Multiplication is defined by  
 \[t_i t_j = \binom {i+j}j t_{i+j} \text{  if  } i+j< p^n, \text{  and
   zero otherwise} \] 
and $ t_0$ is the identity.  The coalgebra structure is given by 
\[ \Delta(t_r) = \sum_{j=0}^r t_j \otimes t_{r-j} \text{  and  }  \epsilon(t_r)  = \delta_{0, r} .\]
The antipode  is given by $s(t_r) = (-1)^rt_r$.   \end{definition}

\begin{remark}  
As Alan Koch has pointed out to us, the divided power Hopf algebra
$\mathcal{A}(n)$ represents the group scheme given by the kernel of
the Frobenius homomorphism on the additive group of Witt vectors of
length $n$ over $K$.
\end{remark}  

Let $L = K(x)$ be a primitive purely inseparable field extension of
exponent $n$.  Let $\mathcal{A}(n)$ act on $L$ by
\begin{equation} \label{t-action} 
 t_r(x^s) = \binom sr x^{s-r} 
\end{equation}
(where $\binom sr = 0$ if  $r > s$). Then, as Sweedler
observes \cite[p.~215]{Sw69}, $L$ is an $ \mathcal{A}(n)$-module algebra,
$K$ is the field of constants  
\[L^{\mathcal{A}(n)} = \{ y \in L :  h(y) = \epsilon(h)y  \text{ for
  all } h  \in \mathcal{A}(n)\} \]
and $[\mathcal{A}(n):K] = [L:K] = p^n$.  By a theorem of Sweedler
 (\cite[Theorem 10.1.1]{Sw69}),  these conditions imply that the map from
 $L \otimes_K \mathcal{A}(n)$ to $\End_K(L)$ is bijective,  and hence
 $L$ is an $\mathcal{A}(n)$-Hopf Galois extension of $K$.  

 This example also shows up in \cite[Lemma 1.2,3]{AS69} and in dual
 form (that is, $L$ is an $\mathcal{A}(n)^*$-Galois object) in
 \cite[Example 4.11]{CS69}. 
 
 Evidently if $\mathcal{A}(n)$ acts on $L/K$, then the basis $\{ t_0,
 t_1, \ldots, t_{p^n-1}\}$ of $\mathcal{A}(n)$  defines a higher
 derivation of $L/K$.  So henceforth we denote $t_i$ by $D_i$.  As a
 $K$-vector space,  
 \[ \mathcal{A}(n) = K[D_0, D_1, \ldots D_{p^n-1}], \]
and we turn our attention now to
the structure of $\mathcal{A}(n)$ as an algebra.

\begin{proposition}\label{dpHa} As a $K$-algebra, 
\[ \mathcal{A}(n) = K[D_1, D_p, \ldots D_{p^{n-1}}]\cong K[T_0, T_1,
  \ldots T_{n-1}]/(T_0^p, T_1^p, \ldots, T_{n-1}^p]  \]
is an exponent $p$ truncated polynomial algebra over $K$.  
\end{proposition}

To show this, it is convenient to invoke
\begin{theorem}[Lucas's Theorem, 1878]  \label{lucas}
Let $p$ be prime, and let the integers $a, b \ge 0$ be written
$p$-adically: 
\[ \begin{aligned} a &= a_0 + a_1p + \ldots + a_rp^r\\
b &= b_0 + b_1p + \ldots + b_rp^r  \end{aligned}\]
where we may assume $b \le a$ (so that $b_r $ may be $0$) and  $0 \le a_i, b_i < p$ for all $i$.  Then
\[ \binom ab \equiv \binom {a_0}{b_0} \binom {a_1}{b_1} \cdots \binom {a_r}{b_r}  \pmod{p}. \]
\end{theorem}
Here $\binom 00 = 1$ and $\binom ab = 0$ if $a < b$. 
For a nice proof of Lucas's Theorem using the Binomial Theorem modulo
$p$, see \cite{Rid13}.  

\begin{proof} [Proof of Proposition \ref{dpHa}] 

This is a matter of showing by induction (using Lucas's Theorem) that modulo $p$:  first,  
\[  k!D_{kp^r} = D_{p^r}^k \]
 for $1 \le k \le p$, and hence $ D_{p^r}^p = 0$ for all $r$; and second, 
\[ D_{a_0 + a_1p + \ldots a_rp^r} = D_{a_0}D_{a_1p} \cdots D_{a_rp^r}. \]
Hence 
\[  D_{a_0 + a_1p + \ldots a_rp^r}  = \frac {D_1^{a_0}}{a_0!} \cdot  \frac {D_p^{a_1}}{a_1!} \cdots  \frac {D_{p^r}^{a_r}}{a_r!} .\]
Since all the factorials are units modulo $p$, the result follows.
\end{proof}

\subsection{$\mathcal{A}(n)$-scaffolds on purely inseparable extensions}

Let $K$ be a local field with normalized valuation $v_K$ and uniformizing
parameter $\pi$. Let $L$ be a purely inseparable field
extension of exponent $n$ which is totally ramified. To see that $L/K$
is primitive, let $\nu$ be a uniformizing parameter for $L$. As $L/K$
is purely inseparable, we have $\nu^{p^n} \in K$. As $L/K$ is totally
ramified, it follows that $v_K(\nu^{p^n})=1$, so that $\nu^{p^{n-1}}
\not \in K$ and $L=K(\nu)$. Now that $L/K$ is primitive,
$\mathcal{A}(n)=K[D_1, D_p, \ldots, D_{p^{n-1}}]$, the divided power
Hopf algebra of dimension $p^n$, acts on $L/K$. Indeed, it can act
on $L/K$ in many ways.

Let $0 < b < p^n$ with $\gcd(b,p)=1$, and set $x=\nu^{-b}$. Then
$L=K(x)$ and $v_L(x)=-b$. We specify that $\mathcal{A}(n)$ acts on $L$
as in (\ref{t-action}), that is,
\begin{equation} \label{E: act}  
   D_{p^r}(x^a) = \binom a{p^r}x^{a-p^r} .
\end{equation}
By Lucas's Theorem,
\[ \binom a{p^r} =a_{(r)}\]
where $a_{(r)}$ is the $r$th digit  of the $p$-adic expansion of $a$.  
Note that this action depends on the choice of the generator $x$ for
$L/K$, and therefore depends on $b$. 

\begin{remark}\label{intuit-connect}
At this point, we make explicit the connection with the intuition of a
scaffold, presented in \S\ref{intro}. Let $X_i=x^{p^{n-i}}$,
and $\Psi_i=D_{p^{n-i}}$.  Then \eqref{E: act} and \eqref{deriv-exact}
agree where $0\leq a<p^n$ is expressed $p$-adically as both
$a=a_0+a_1p+\cdots +a_{n-1}p^{n-1}$ to be consistent with Theorem
\ref{lucas} and $a=a_{(0)}+a_{(1)}p+\cdots +a_{(n-1)}p^{n-1}$ to be
consistent with \eqref{first-p-adic}. 
\end{remark}

\begin{remark}\label{p-adic_convention}
In the present purely inseparable situation, the convention adopted by
\eqref{first-p-adic} where integers $0\leq a<p^n$ are expressed
$p$-adically as $a=\sum_{i=1}^na_{(n-i)}p^{n-i}$ (and thus necessarily
$X_i=x^{p^{n-i}}$ and $\Psi_i=D_{p^{n-i}}$) may seem awkward. So it
is worth reiterating that we adopted this convention because
scaffolds arose first in the setting of Galois extensions and in that
setting it is natural to label the $i$th ramification break with the
subscript $i$.
\end{remark}

To complete the proof of Theorem \ref{insep} we now define an
$\mathcal{A}(n)$-scaffold on $L/K$ with $b$ as its sole shift
parameter. This means, following Definition 2.3, that we need two sets
of elements: elements $\lambda_t \in L$ for all integers $t$ with
$v_L(\lambda_t) = t$, and elements $\Psi_k \in \mathcal{A}(n)$ for $1
\le k\le n$.  But since the shift parameters are all the same, we can
simplify notation.  From \eqref{eub}, we see that $\eub(s)=bs$.  Let
$a$ be an integer with $ab \equiv -1 \pmod{p^n}$.  Thus for
$t\in\mathbb{Z}$, $\eua(t)$ can be more easily understood as the least
non-negative residue of $at$ modulo $p^n$.  For each $t$ in
$\mathbb{Z}$, define $f_t$ by
\[ t = -b\eua(t) + p^n f_t .\]
Hence $f_t > 0$ for all $t>0$.  Expand $\eua(t)$ $p$-adically as $\eua(t) =
\eua(t)_{(0)} + \eua(t)_{(1)}p + \ldots + \eua(t)_{(n-1)}p^{n-1}$.  For the
$\lambda_t$ with $t\in\mathbb{Z}$, set
\[ \lambda_t = \frac {\pi^{f_t} x^{\eua(t)}}{\eua(t)_{(0)}!\eua(t)_{(1)}!\cdots
  \eua(t)_{(n-1)}!}=\pi^{f_t}\prod_{i=1}^{n}\frac{X_i^{\eua(t)_{(n-i)}}}{\eua(t)_{(n-i)}!},
\]  
where $X_i$ is as in Remark \ref{intuit-connect}.  Observe that $f_t$
was defined so that $v_L(\lambda_t) = -b\eua(t) + p^n f_t = t$, and if
$t_1 \equiv t_2 \pmod{p^n}$, then $\eua(t_1) = \eua(t_2)$ and
$\lambda_{t_1} \lambda_{t_2}^{-1} = \pi^{f_{t_1} - f_{t_2}}\in K$.  As
in Remark \ref{intuit-connect}, for $1\leq r \leq n$, set
 $$\Psi_r = D_{p^{n-r}}.$$ 
Now observe that 
\eqref{E: act}  together with Lucas's Theorem imply
$$ \Psi_r\lambda_t = \begin{cases}
      \lambda_{t + p^{n-r}b} &  \mbox{if } \eua(t)_{(n-r)} > 0, \\
       0 & \mbox{if } \eua(t)_{(n-r)} =  0. 
   \end{cases} $$
Based upon
Definition \ref{tol-scaffold}
this justifies our assertion that the intuition of a scaffold
yields a scaffold. It also proves.

\begin{proposition}The elements  $\{ \lambda_t\}_{t
  \in \mathbb{Z}}, \{\Psi_r \}_{0 \le r \le n}$ form an
  $\mathcal{A}(n)$-scaffold on $L$ of precision $\infty$.
\end{proposition}

\appendix 

\section{Comparison of definitions of scaffold}

\subsection{An alternative characterization of $A$-scaffolds}

Let $K$ be a local field with residue characteristic $p$,
let $L/K$ be a totally ramified field extension of degree $p^n$, and
let $A$ be a $K$-algebra of dimension $p^n$ with a $K$-linear
action on $L$. We assume that we are given a family of elements $\Psi_1, \ldots,
\Psi_n$ of $A$. For $s \in \Spn$, we then have the set
$\Upsilon^{(s)}$ of monomials in the $\Psi_i$, as defined before
(\ref{def-Psi}). We also suppose we are given functions $\eub$,
$\eua$ corresponding to a family of shift parameters $b_1, \ldots,
b_n$, all relatively prime to $p$. We consider the following
conditions on the $\Psi_i$:
\begin{equation} \label{A-aug}
 \Psi_i \cdot 1 =0  \mbox{ for each } i;
\end{equation}
\begin{equation} \label{A-eub-shift}
  v_L(\Psi \cdot \rho) = v_L(\rho)+\eub(s) \mbox{ for all } \Psi \in
  \Upsilon^{(s)} \mbox{ and }   s \in \Spn,  
\end{equation}
for some given $\rho \in L\backslash\{0\}$;
\begin{equation} \label{A-p-power}
  v_L( \Psi_i^p \cdot \alpha) > v_L(\alpha)+b_ip^{n-i+1} 
 \mbox{ for all } i \mbox{ and all } \alpha \in L\backslash\{0\}; 
\end{equation}
and the stronger form of (\ref{A-p-power}), 
\begin{equation} \label{A-p-triv}
  \Psi_i^p = 0  \mbox{ for all } i.
\end{equation}
Let $\{\lambda_t\}_{t \in \bZ}$ be any family of elements of $L$
  satisfying the conditions of Definition \ref{tol-scaffold}(i):
  $v_L(\lambda_t)=t$ for all $t$, and $\lambda_{t_1}
  \lambda_{t_2}^{-1} \in K$ whenever $t_1 \equiv t_2 \pmod{p^n}$. 

\begin{theorem} \label{app-thm} 
\ \\ \noindent (i) Suppose that the $\Psi_i$ satisfy (\ref{A-aug}) and
(\ref{A-p-power}), and there is some $\rho$ for which
(\ref{A-eub-shift}) holds. Then the $\lambda_t$ and the $\Psi_i$ form
an $A$-scaffold of precision $\tol =1$ on $L$ in the sense of Definition
\ref{tol-scaffold}, and its shift parameters are $b_1, \ldots, b_n$.
Moreover $\eua(v_L(\rho))=p^n-1$.

\noindent (ii) If, furthermore, (\ref{A-p-triv}) holds and 
$A$ is commutative, then the $\lambda_t$ may
be chosen so that the $A$-scaffold has precision $\infty$.

\noindent (iii) Conversely, if the $\lambda_t$ and the $\Psi_i$
form an $A$-scaffold of some precision $\tol \geq 1$ in the
sense of Definition \ref{tol-scaffold}, then (\ref{A-aug}) and
(\ref{A-p-power}) hold, and (\ref{A-eub-shift}) holds for any $\rho
\in L$ with $\eua(v_L(\rho))=p^n-1$.
\end{theorem}
\begin{proof}
(i) Since (\ref{A-aug}) holds by hypothesis, we will have an
  $A$-scaffold of precision $1$ provided that the
  congruence in Definition \ref{tol-scaffold}(ii) holds with $\tol=1$
  and some choice of the units $u_{i,t}$. This will be the case if, for
  each $t \in \bZ$ and each $i$, we have 
 \begin{equation} \label{A-val-shift-final}
   v_L(\Psi_i \cdot \lambda_t) \quad  
    \begin{cases} \quad = t + p^{n-i}b_i &\mbox{ if } \eua(t)_{(n-i)} \geq 1,\\ 
                 \quad > t + p^{n-i}b_i &\mbox{ if } \eua(t)_{(n-i)} =
                 0. \end{cases} 
\end{equation} 

Fix $i$, and, for each $s \in \Spn$, define $\Psi_*^{(s)} \in
  \Upsilon^{(s)}$ by
$$ \Psi_*^{(s)} = \Psi_i^{s_{(n-i)}} \Psi_{n}^{s_{(0)}} \ldots
\Psi_{i-1}^{s_{(n-i-1)}} \Psi_{i+1}^{s_{(n-i+1)}} 
\ldots \Psi_1^{s_{(n-1)}}. $$ 
Thus $\Psi_*^{(s)}$ is obtained
from $\Psi^{(s)}$ by bringing all the factors $\Psi_i$ to the left (so
in particular $\Psi_*^{(s)} = \Psi^{(s)}$ if $i=n$). From
(\ref{A-eub-shift}) we have $v_L(\Psi_*^{(s)} \cdot \rho)=v_L(\rho)+\eub(s)$.
Thus $\{ v_L(\Psi_*^{(s)} \cdot \rho) : s \in \Spn\}$ is a complete set of
residues mod $p^n$, and hence $\{ \Psi_*^{(s)} \cdot \rho : s \in \Spn\}$ is
a $K$-basis for $L$. 

Now fix $s$ as well. If $s_{(n-i)}<p-1$ then $\Psi_i \Psi_*^{(s)} \in
\Upsilon^{(s')}$ where $s'=s+p^{n-i}$, so that
$s'_{(n-i)}=s_{(n-i)}+1$ and $s'_{(n-j)}=s_{(n-j)}$ 
for $j \neq i$. Thus, from (\ref{A-eub-shift}), 
\begin{equation} \label{Psi-star1}
   v_L(\Psi_i \Psi_*^{(s)}\cdot \rho) = v_P(\rho) + \eub(s')
   = v_L(\Psi_*^{(s)} \cdot \rho) +b_i p^{n-i} 
       \mbox{ if } s_{(n-i)}<p-1. 
\end{equation}
On the other hand, if $s_{(n-i)}=p-1$ then $\Psi_i
\Psi_*^{(s)}=\Psi_i^p \Psi_*^{(s'')}$ where $s''=s-(p-1)p^{n-i}$, so
that $s''_{(n-i)}=0$ and $s''_{(n-j)}=s_{(n-j)}$ for $j \neq i$.  
Then $v_L(\Psi_i \Psi_*^{(s)}\cdot\rho) = 
 v_L(\Psi_i^p \Psi_*^{(s'')} \cdot \rho) > v_L(\Psi_*^{(s'')} \cdot \rho)
 +b_i p^{n-i+1}$ by (\ref{A-p-power}). But $v_L(\Psi_*^{(s'')}\cdot\rho) = 
v_L(\rho)+\eub(s'')=v_L(\rho) + \eub(s) - b_i (p-1) p^{n-i}$ by
(\ref{A-eub-shift}). Hence 
\begin{equation} \label{Psi-star2}
  v_L(\Psi_i \Psi_*^{(s)}\cdot\rho)  > v_L(\Psi_*^{(s)}\rho) +b_i p^{n-i} 
       \mbox{ if } s_{(n-i)}=p-1.
\end{equation}

Now let $\alpha \neq 0$ be an arbitrary element of $L$ with
$v_L(\alpha)=t$. Then we may write $\alpha = \sum_{s \in \Spn} x_s
\Psi_*^{(s)} \cdot \rho$ with the $x_s \in K$. The sum contains a
unique term of minimal valuation; let this occur at $s=s'$. Then
$t=p^n v_K(x_{s'})+v_L(\rho)+\eub(s')$. Applying (\ref{Psi-star1}) or
(\ref{Psi-star2}) to each term in the sum separately, we obtain
\begin{equation} \label{A-shift-alpha}
   v_L(\Psi_i \cdot \alpha) \quad  
    \begin{cases} \quad = t + p^{n-i}b_i &\mbox{ if } s'_{(n-i)} <p-1,\\ 
                 \quad > t + p^{n-i}b_i &\mbox{ if } s'_{(n-i)} =p-1.
        \end{cases} 
\end{equation} 

Before completing the proof of (\ref{A-val-shift-final}), we consider
$v_L(\rho)$. There is some $s \in \Spn$ for which $v_L(\Psi_*^{(s)} \cdot
\rho) \equiv 0 \pmod{p^n}$. By (\ref{A-eub-shift}), $s$ is independent of $i$. 
We may write $\Psi_*^{(s)} \cdot \rho = x + \beta$ where
$x \in K$ and $v_L(\beta)>v_L(\Psi_*^{(s)} \rho)$. Using (\ref{A-aug})
and (\ref{A-shift-alpha}), we therefore have 
$$ v_L(\Psi_i \Psi_*^{(s)} \cdot \rho) = v_L(\Psi_i \cdot \beta) > 
                   v_L(\Psi_*^{(s)} \cdot \rho) + p^{n-i}b_i . $$
Comparing with (\ref{A-shift-alpha}), we see that we must have $s_{(n-i)}=p-1$. 
This holds for each $i$, so $s=p^n-1$. Thus, by the choice of $s$, we
have $v_L(\rho) \equiv -\eub(p^n-1) \pmod{p^n}$, or equivalently,
$\eua(v_L(\rho))=p^n-1$. 

Recall that in (\ref{A-shift-alpha}) we have $t=v_L(\alpha)  \equiv v_L(\rho)
+\eub(s') \pmod{p^n}$. Thus $t \equiv
-\eub(p^n-1)+\eub(s')=-\eub(p^n-1-s')$. Hence
$\eua(t)=p^n-1-s'$, so the condition $s'_{(n-i)} <p-1$ in
(\ref{A-shift-alpha}) is equivalent to $\eua(t)_{(n-i)} \geq 1$. 
Now (\ref{A-val-shift-final}) follows on applying
(\ref{A-shift-alpha}) to $\alpha=\lambda_t$. 

(ii) Fix a uniformizing element $\pi$ of $K$.  
Given $t \in \bZ$, we choose $\lambda_t= \pi^{f} \Psi^{(s)} \cdot
\rho$ where $p^n f + 
v_L(\rho)+\eub(s)=t$.
 Then $v_L(\lambda_t)=t$, and $\lambda_{t_1}
  \lambda_{t_2}^{-1} = \pi^{(t_1-t_2)/p^n} \in K$ when $t_1 \equiv t_2
  \pmod{p^n}$. Also, 
as we have shown above, $\eua(t)=p^n-1-s$, so that $\eua(t)_{(n-i)} \geq 1$ if
and only if $s_{n-i)}<p-1$.  
As $A$ is commutative and $\Psi_i^p=0$, we have $\Psi_i \Psi^{(s)}
\cdot \rho =0$ if 
  $s_{(n-i)}=p-1$, so that $\Psi_i\cdot  \lambda_t=0$. On the other hand, if
    $s_{(n-i)} \neq p-1$ then $\Psi_i \cdot \lambda_t = \lambda_{t+b_i
      p^{n-i}}$. Thus the congruence in Definition
    \ref{tol-scaffold}(ii) becomes an equality (with $u_{i,t}=1$ for
    all $i$ and $t$). 

(iii) Since we have an $A$-scaffold in the sense of Definition
    \ref{tol-scaffold}, (\ref{A-aug}) holds. Also,  from
    (\ref{part-graded-val}), for any $s \in \Spn$ and
any $\Psi \in \Upsilon^{(s)}$ we have
\begin{equation} \label{A-psi-lambda}
   v_L(\Psi \cdot\lambda_t) \quad  
    \begin{cases} \quad = t + \eub(s) &\mbox{ if } s \preceq \eua(t),\\ 
                 \quad > t + \eub(s) &\mbox{ otherwise}.
        \end{cases} 
\end{equation}
For an arbitrary $\alpha \in L$ with $v_L(\alpha)=t$, we may write
$\alpha = u \lambda_t + \sum_{j=1}^{p^n-1} y_j \lambda_{t+j}$ for some
$u \in \euO_K^\times$ and $y_j \in \euO_K$. Applying
(\ref{A-psi-lambda}) to each term, we find that (\ref{A-psi-lambda})
still holds if we replace $\lambda_t$ by $\alpha$. In particular,
taking $\Psi=\Psi_i$, we have
\begin{equation} \label{A-psi-alpha}
   v_L(\Psi_i \cdot \alpha) \quad  
    \begin{cases} \quad = t + b_i p^{n-i} &\mbox{ if } \eua(t)_{(n-i)}>0,\\ 
                 \quad > t + p^{n-i}b_i &\mbox{ otherwise.}
        \end{cases} 
\end{equation}
Moreover, writing $t'=v_L(\Psi_i \cdot \alpha)$, we have
$\eua(t')_{(n-i)}=\eua(t)_{(n-i)}-1$ if $\eua(t)_{(n-i)}>0$. Repeating
this argument $p$ times, we obtain (\ref{A-p-power}). Finally, for
any $\rho$ with 
$\eua(v_L(\rho))=p^n-1$, (\ref{A-eub-shift}) follows inductively from 
(\ref{A-psi-alpha}).
\end{proof}

\subsection{Galois scaffolds in previous papers}

We now use Theorem \ref{app-thm} to explain how the $A$-scaffolds of
this paper are related to the Galois scaffolds of the earlier papers
\cite{elder:scaffold,elder:sharp-crit,byott:scaffold}. There we
considered only abelian extensions $L/K$ in characteristic $p$; 
the extensions in \cite{elder:scaffold,byott:scaffold} were elementary
abelian of arbitrary rank, and those in \cite{elder:sharp-crit} were
elementary abelian or cyclic of degree $p^2$. The algebra $A$ 
acting on $L$ was always the group algebra $A=K[G]$ with
$G=\Gal(L/K)$; in this setting, (\ref{A-aug}) simply
says that the $\Psi_i$ lie in the augmentation ideal of $K[G]$.

The definition of Galois scaffold varies slightly between these
papers, and the conditions explicitly required are a little less
restrictive than those of this paper. The Galois scaffolds constructed
turn out to satisfy supplementary conditions which were used
in obtaining results on Galois module structure. For the reader's
convenience, the role of the different conditions in the various
papers is summarized in Table 2. (The conditions not already mentioned
are introduced below.)

\begin{table} \label{scaff-props}  
\centerline{ 
\begin{tabular}{l | l | l} 
   Paper  &  Explicit in  &  Used for Galois \\   
          & definition    & module structure \\ \hline
\cite{elder:scaffold} & (\ref{all-res})  & \\ 
\cite{elder:sharp-crit} & (\ref{A-aug}), (\ref{all-res}),
(\ref{j-shift})   & 
  (\ref{A-aug}), (\ref{A-eub-shift}), (\ref{A-p-power}) \\
\cite{byott:scaffold} & (\ref{A-aug}), (\ref{all-res}), (\ref{j-shift}),
& (\ref{A-aug}), (\ref{A-eub-shift}), (\ref{A-p-triv}) 
\end{tabular}
}  
\vskip5mm

\caption{Properties of Galois scaffolds.} 
\end{table}

\subsubsection{The paper \cite{elder:scaffold}}
Galois scaffolds first appeared in \cite{elder:scaffold}, where they
were presented as a strengthening of the valuation criterion. Let $K$ 
be a local field of residue characteristic $p>0$, and
let $L/K$ be a totally ramified Galois extension of degree $p^n$ with
Galois group $G=\Gal(L/K)$. We say that $L/K$ satisfies the valuation
criterion if there exists $c \in \bZ$ such that $L=K[G] \cdot \rho$ for every
$\rho \in L$ with $v_L(\rho)=c$. In \cite{elder:scaffold}, $L/K$ was
said to have a Galois scaffold if there exist $c \in \bZ$ and elements
$\Psi_1, \ldots, \Psi_n \in K[G]$ such that, for every $\rho \in L$
with $v_L(\rho)=c$, the following condition holds: 
\begin{equation} \label{all-res}
\left\{ v_L\left(\Psi^{(s)} \cdot \rho\right)  
 :  s \in \Spn \right \} \mbox{ is a complete set of residues } p^n. 
\end{equation}
Since $L/K$ is totally ramified,
(\ref{all-res}) implies the valuation criterion for $L/K$.

For the Galois scaffolds actually constructed in \cite{elder:scaffold},
the conditions (\ref{A-aug}) and (\ref{A-eub-shift}) hold, where the
shift parameters $b_i$ used to define $\eub$ are the (lower) ramification breaks
of $L/K$,  
and $\rho$ is any element of $L$ with $v_L(\rho) \equiv b_n
\pmod{p^n}$. (Indeed, for the near one-dimensional extensions
considered in \cite{elder:scaffold}, the $b_i$ are all congruent mod $p^n$, so
$\eub(s) \equiv b_ns \pmod{p^n}$.) Moreover,
(\ref{A-p-triv}) holds, and Theorem \ref{app-thm}(i)(ii)
shows that the Galois scaffolds in \cite{elder:scaffold} are
$K[G]$-scaffolds of precision $\infty$ in the sense of this paper. 

\subsubsection{The paper \cite{byott:scaffold}}
In \cite{byott:scaffold} (which was in fact written before
\cite{elder:sharp-crit}), the definition of Galois scaffold was 
refined to require (\ref{A-aug}) explicitly, and also to require 
the uniformity condition 
\begin{equation} \label{j-shift} 
 v_L(\Psi_i^j \cdot \rho)-v_L(\rho) = j \cdot (v_L(\Psi_i
 \cdot\rho')-v_L(\rho') )  
 \end{equation}
whenever $0 \leq j \leq p-1$ and $v_L(\rho)$, $v_L(\rho') \equiv c
\pmod{p^n}$. Here, as above, $c$ is the integer occurring in the valuation
criterion. Note that (\ref{j-shift}) makes no explicit mention of the
ramification breaks $b_i$. If we set
$$ a_i = v_L(\Psi_i \cdot \rho) -v_L( \rho), $$
then (\ref{j-shift}) means that $a_i$ is independent of the choice of
$\rho$ with $v_L(\rho) \equiv c \pmod{p^n}$, and that 
\begin{equation} \label{a-shift}
   v_L(\Psi_i^j \cdot \rho) = v_L(\rho) + ja_i \mbox{ if } 0\leq j \leq p-1
  \mbox{ and } v_L(\rho)\equiv c \pmod{p^n}.  
\end{equation}
Moreover, if (\ref{A-eub-shift}) holds for the function $\eub$ given
by some shift parameters $b_1, \ldots, b_n$, then (\ref{a-shift}) holds for
$a_i=p^{n-i} b_i$ and any $c \equiv -\eub(p^n-1) \pmod{p^n}$.  

In view of (\ref{A-eub-shift}), it is reasonable to replace
(\ref{j-shift}) by 
\begin{equation} \label{a-expand}
                v_L\left( \Psi^{(s)} \cdot \rho\right) 
        = v_L(\rho)
        + \sum_{i=1}^n s_{(n-i)} a_i \mbox{ for all } s \in \Spn,
\end{equation}
where again $\rho$ is any element of $L$ with $v_L(\rho) \equiv c
\pmod{p^n}$. Now if (\ref{a-expand}) holds for some integers $a_i$,
then, by Proposition \ref{B-function} below, (\ref{all-res}) is
equivalent to the condition that (possibly after renumbering the
$\Psi_i$ and the $a_i$) there are integers $b_1, \ldots, b_n$, all
relatively prime to $p$, such that $a_i=p^{n-i}b_i$. If we use these
to define the function $\eub$, then, in the case of an abelian
extension, (\ref{a-expand}) is equivalent to (\ref{A-eub-shift}). We
may therefore regard (\ref{A-eub-shift}) as a natural strengthening of
(\ref{a-expand}), and hence of (\ref{j-shift}). 

The extensions considered in \cite{byott:scaffold} are the near
one-dimensional extensions constructed in \cite{elder:scaffold}, and
the Galois scaffolds used are those of that paper. As explained above,
they satisfy (\ref{A-aug}), (\ref{A-eub-shift}) and
(\ref{A-p-triv}). These properties were used in \cite{byott:scaffold}
to investigate the Galois module structure of the valuation rings.

\subsubsection{The paper \cite{elder:sharp-crit}}
So that we focus on those results in \cite{elder:sharp-crit} which
are in neither \cite{elder:scaffold} nor \cite{byott:scaffold}, we
restrict our discussion here to cyclic extensions of degree $p^2$. In
any case, \cite{elder:sharp-crit} used the same definition of Galois
scaffold as \cite{byott:scaffold}. The Galois scaffolds considered in
\cite{elder:sharp-crit} satisfy (\ref{A-aug}) and
(\ref{A-eub-shift}). The cyclic Galois scaffolds 
satisfy $\Psi_1^p=\Psi_2$, $\Psi_2^p=0$, and \eqref{A-p-power} holds
since $b_2>p^2 b_1$. Thus again they are $K[G]$-scaffolds of some
precision $\tol \geq 1$. (In fact $\tol = b_2-pb_1$.) These properties
are used in \cite{elder:sharp-crit} to investigate the Galois module
structure of the valuation ring in cyclic extensions of degree $p^2$
admitting a Galois scaffold.

\subsection{An alternative form of the function $\eub$}

In the above discussion, we needed the following result:

\begin{proposition}\label{B-function}
Given $a_i\in\bZ$, let $\eub':\mathbb{S}_p^n\longrightarrow\bZ$ be
defined by
$$\eub'(x_1,\ldots ,x_n)=\sum_{i=1}^na_ix_i.$$
Let $r:\bZ\longrightarrow\Spn$
be given by $r(a)\equiv a\pmod{p^n}$, and let $v$ denote the normalized valuation on the $p$-adic rationals.
Then the function $r\circ \eub' : \Spn \to \Spn$ is bijective
if and only if, after relabelling if necessary,
$v(a_i)=n-i$ for $1\leq i\leq n$.
\end{proposition}
\begin{proof}
Note that $r\circ \eub'$ is surjective if and only if it is bijective.  If
$v(a_i)=n-i$ for $1\leq i\leq n$, then clearly the image of $r\circ
\eub'$ is $\mathbb{S}_{p^n}$.  So consider the converse: Let
$\eub'_n(x_1,\ldots ,x_n)=\sum_{i=1}^na_ix_i$ and induct on $n$. For
$n=1$ the statement holds, since
$r\circ\eub'_1:\mathbb{S}_p\longrightarrow\mathbb{S}_p$ is bijective if
and only if $\gcd(a_1,p)=1$. Assume the statement holds for $n-1$, and
consider it for $n$.  If $v(a_i)\geq 1$ for all $i$, then each
$\eub'_n(x_1,\ldots, x_n)$ is a multiple of $p$. So we may assume there
is an $a_i$ that is relatively prime to $p$. Relabel so that
$v(a_n)=0$.

We prove now that, for $0\leq k\leq p^{n-1}-1$, there exist
$x_{i,k}\in\mathbb{S}_p$ with  
\begin{equation}\label{multp}
kpa_n=\eub_n'(x_{1,k},\ldots ,x_{n-1,k},0)=x_{1,k}a_1+\cdots +x_{n-1,k}a_{n-1}.
\end{equation}
The case $k=0$ is clear. Assume the statement holds for $k-1$, and
consider it for $k$.
Since $r(kpa_n)\in\mathbb{S}_{p^n}$ and $r\circ\eub'_n$ is surjective,
there are $x_{i,k}\in\mathbb{S}_p$ such that 
$kpa_n=x_{1,k}a_1+x_{2,k}a_2+\cdots +x_{n,k}a_{n}$.
If $x_{n,k}\neq 0$ then, on subtracting 
$(k-1)pa_n=x_{1,k-1}a_1+\cdots +x_{n-1,k-1}a_{n-1}$, we find that
\begin{equation}\label{cont-inj}
(p-x_{n,k})a_n=(x_{1,k}-x_{1,k-1})a_1+\cdots +(x_{n-1,k}-x_{n-1,k-1})a_{n-1}.
\end{equation}
Let $y_n=p-x_{n,k}$, and, for $1 \leq i\leq n-1$, let
$$y_i=\begin{cases}x_{i,k-1}-x_{i,k}& \mbox{if }x_{i,k-1}\geq x_{i,k},\\
0& \mbox{if }x_{i,k-1}< x_{i,k}.
\end{cases}$$
Let $z_n=0$, and
for $1 \leq i\leq n-1$ let 
$$z_i=\begin{cases}0& \mbox{if }x_{i,k-1}\geq x_{i,k},\\
x_{i,k}-x_{i,k-1}& \mbox{if }x_{i,k-1}< x_{i,k}.
\end{cases}$$
Then \eqref{cont-inj} means that $\eub_n'(y_1,\ldots
,y_n)=\eub_n'(z_1,\ldots, z_n)$. As $y_n \neq z_n$, this contradicts the
injectivity of $r \circ \eub_n'$. Thus
$x_{n,k}=0$, and the statement holds for $k$.

Since $r\circ\eub_n$ is injective, \eqref{multp} establishes a bijection
between $\mathbb{S}_p^{n-1}$ and the multiples of $p$ in
$\mathbb{S}_{p^n}$. In particular, $a_i = \eub'_n(0,\ldots,
0,1,0\ldots 0)$ is a multiple of $p$ for each $i\leq n-1$.  Thus the
image of $\mathbb{S}_p^{n-1}$ under $\eub'_{n-1}(x_1,\ldots,
x_{n-1})=\sum_{i=1}^{n-1}x_i(a_i/p)$ maps modulo $p^{n-1}$ onto
$\mathbb{S}_{p^{n-1}}$. Using induction, we may relabel so that
$v(a_i/p)=n-1-i$ for $1\leq i\leq n-1$.  We conclude that $v(a_i)=n-i$ for
$1\leq i\leq n$.
\end{proof}

\bibliography{BCE-aif}

\end{document}